\documentclass[10pt]{article}

\usepackage{amsmath, amsfonts, amsthm, amssymb, amsrefs}
\usepackage{mathtools}
\usepackage{tikz, graphicx, pgfplots, color, xcolor}
	\pgfplotsset{compat=1.12} % indicates version of pgfplots used for consistency
	\usetikzlibrary{positioning}
\usepackage{enumitem}
\usepackage[normalem]{ulem}	% allows sout (strikeout) of text

\usepackage[left=1.2in,right=1.2in,top=1in,foot=1in]{geometry}
\allowdisplaybreaks		% in mathmode environment
\usepackage[linkbordercolor=cyan, pdfborder={0 0 0.5}]{hyperref}
\newcommand{\benum}[2]{
	\begin{enumerate}[
		leftmargin={#2}, label={#1}
		]
	}
\usepackage{amsthm}

	\theoremstyle{definition}

	\newtheoremstyle{TheoremNum}
        {\topsep}{\topsep}              %%% space between body and thm
        {\itshape}                      %%% Thm body font
        {}                              %%% Indent amount (empty = no indent)
        {\bfseries}                     %%% Thm head font
        {.}                             %%% Punctuation after thm head
        { }                             %%% Space after thm head
        {\thmname{#1}\thmnote{ \bfseries #3}}%%% Thm head spec
    \theoremstyle{TheoremNum}

	\newcommand{\df}[1]{{\bf\emph{#1}}}		% definitions within definition (bold)
\renewcommand{\epsilon}{\varepsilon}	% varepsilon
\renewcommand{\phi}{\varphi}			% varphi
\renewcommand{\hat}[1]{\widehat{#1}}	% hat wide

\definecolor{orange}{RGB}{250, 140, 0}
	
\definecolor{turq}{RGB}{0, 160, 160}

\newcommand\mc[1]{\mathcal{#1}}

\newcommand{\rr}{\ensuremath{\mathbb{R}}}

\newcommand{\zz}{\ensuremath{\mathbb{Z}}}

\usetikzlibrary{cd, arrows, automata}
	\newcommand*{\arrow}[1]{\arrow[#1]}

\usepackage{chemfig}
\newcommand{\diag}{diag}		% diagonal
		\DeclareFontFamily{U}{mathx}{\hyphenchar\font45}
		\DeclareFontShape{U}{mathx}{m}{n}{
	      <5> <6> <7> <8> <9> <10>
	      <10.95> <12> <14.4> <17.28> <20.74> <24.88>
	      mathx10
	      }{}
		\DeclareSymbolFont{mathx}{U}{mathx}{m}{n}
		\DeclareFontSubstitution{U}{mathx}{m}{n}
		\DeclareMathAccent{\widecheck}{0}{mathx}{"71}
\newcommand{\alphab}{\boldsymbol{\alpha}}%
\newcommand{\betab}{\boldsymbol{\beta}}%{\text{\boldmath$\beta$}}

\newcommand{\varrhob}{\boldsymbol{\varrho}}
\newcommand{\R}{\mathbb{R}}
\newcommand{\N}{\mathbb{N}}

\usepackage[version=4]{mhchem}

\newcommand{\Ab}{\mathbf{A}}

\newcommand{\Bb}{\mathbf{B}}
\newcommand{\bb}{\mathbf{b}}
\newcommand{\cb}{\mathbf{c}}
\newcommand{\eb}{\mathbf{e}}

\newcommand{\Kb}{\mathbf{K}}
\newcommand{\kb}{\mathbf{k}}
\newcommand{\kbh}{\hat{\mathbf{k}}}

\newcommand{\ub}{\mathbf{u}}
\newcommand{\vb}{\mathbf{v}}

\newcommand{\zb}{\mathbf{z}}
\newcommand{\Zb}{\mathbf{Z}}
\newcommand{\xb}{\mathbf{x}}
\newcommand{\Yb}{\mathbf{Y}}
\newcommand{\yb}{\mathbf{y}}
\newcommand{\nulb}{\mathbf{0}}
\newcommand{\oneb}{\mathbf{1}}

\newcommand{\CB}{complex balanced}

\newcommand{\de}{differential equation}
\newcommand{\ikde}{induced kinetic differential equation}
\newcommand{\rhs}{right hand side}

\newcommand{\WR}{weakly reversible}

\newcommand{\SSS}{\mathcal{S}}

\usepackage[version=4]{mhchem}
\theoremstyle{definition}
\newtheorem{lemma}{Lemma}
\newtheorem{theorem}{Theorem}
\newtheorem{proposition}{Proposition}
\newtheorem{example}{Example}
\newtheorem{remark}{Remark}

\newtheorem{conjecture}{Conjecture}
\newtheorem{definition}{Definition}

\newcommand{\vv}[1]{\mathbf{#1}}
\newcommand{\sign}{\mathrm{sign}}
\newcommand{\cf}[1]{\chemfig{#1}}

\newcommand{\NN}{\mathcal{N}}
\newcommand{\T}{^{\top}}

\title{Realizations of kinetic differential equations}
\date{}

\author{%
Gheorghe Craciun
    \thanks{Department of Mathematics and Department of Biomolecular Chemistry, University of Wisconsin--Madison}
\and
Matthew D. Johnston
    \thanks{Department of Mathematics, San Jose State University,
    San Jose, CA 95192, USA}
\and
G\'abor Szederk\'enyi
    \thanks{Faculty of Information Technology and Bionics, P\'azm\'any P\'eter Catholic University,
  Budapest, Hungary}
\and 
Elisa Tonello
    \thanks{Department of Mathematics  and Computer Science,
  Freie Universit\"at,
  Berlin, Germany}
\and 
J\'anos T\'oth
    \thanks{Department of Analysis, Budapest University of Technology and Economics and Laboratory for Chemical Kinetics, E\"otv\"os Lor\'and University,
  Budapest,  Hungary}
\and
Polly Y. Yu
    \thanks{Department of Mathematics, University of Wisconsin-Madison,
  Madison, WI 53706-1325, USA}
}

\begin{document}
%% adjust spacing above/below displaystyle/eqn mode.
%	\abovedisplayskip=10pt
%	\belowdisplayskip=10pt

\maketitle
%\mytoc

%%%%%%%%%%%%%%%%%%%%%%%%%%%%%%%%%%%%%%%%%%%%%%%%%%%%%%%%%%%%%%%%%%%%%%%%%%%%%%%%%%%%%%%%%%%%%%%%%%%%%%%%%%%%%%%%%%%%%%%

\begin{abstract}
The induced kinetic differential equations of a reaction network endowed with mass action type kinetics is a system of polynomial differential equations. The problem studied here is: Given a system of polynomial differential equations, is it possible to find a network which induces these equations; in other words: is it possible to find a \emph{kinetic realization} of this system of differential equations? If yes, can we find a network with some chemically relevant properties (implying also important dynamic consequences), such as reversibility, weak reversibility, zero deficiency, detailed balancing, complex balancing, mass conservation, etc.? The constructive answers presented to a series of questions of the above type are useful when fitting differential equations to datasets, or when trying to find out the dynamic behavior of the solutions of differential equations. It turns out that some of these results can be applied when trying to solve seemingly unrelated mathematical problems, like the existence of positive solutions to algebraic equations. \\

\end{abstract}

\section{Introduction}
%%%%%%%%%%%%%%%%%%%%%%%%%%%%%%%%%%%%%%%%%%%%%%%%
Our goal here is to find reaction networks inducing a given system of 
polynomial differential equations (or classes of equations with coefficients as symbolic parameters) with as many good properties (e.g. weak reversibility, reversibility, small deficiency, small number of linkage classes or reactions, mass conservation, etc.) as possible.

It has been shown that, under mass action kinetics, a necessary and sufficient condition for realizability is a sign structure in the system of polynomial differential equations~\cite{harstoth} (see Lemma~\ref{lemma:hungarian} below). Once we know that a system can arise from a mass action system, the question still lies: can the system arise from a mass action system with good properties? This is what this present work looks at. %\red{(Let us note that \cite{arceojosemarinsanguinomendoza} targeted a similar goal: they looked for kinetic realizations of the differential equations of Biochemical Systems Theory in order to apply recent results from Chemical Reaction Network Theory.)}
There exist similar works outside the realm of mass action kinetics. For example,  \cite{arceojosemarinsanguinomendoza} finds mass action systems that generate the same differential equation as a given S-system or generalized mass action system. The approach allows recent results from Chemical Reaction Network Theory to be applied in Biochemical Systems Theory.

\bigskip

There are several sources of motivation for this problem.
\begin{enumerate}
\item
Having fitted a system of differential equations to  data, one may wonder whether the obtained equations can be interpreted as the mass action type deterministic model of an appropriate reaction network.
\item
There is an internal requirement within this branch of science:
One should like to know as much as possible about the structure of differential equations that arise from modelling a chemical system.
\item
Given a system of polynomial differential equations in any field of pure or applied mathematics, one may wish to have statements on stability or oscillations, similar to those offered by the Horn-Jackson Theorem~\cite{hornjackson}, Zero Deficiency Theorem~\cite{feinbergbook}, Volpert's theorem~\cite{volperthudyaev}, or the Global Attractor Conjecture, where several cases have been proven~\cite{Anderson2011GAC, Pantea2012GAC, GopalkrshnanMillerShiu2014GAC, CraciunNazarovPantea2013GAC}\footnote{%
A proof in full generality has been proposed in \cite{craciunGAH}.}. %
Then it comes in handy to see that the system of differential equations of interest belongs to a well behaving class.
% \item

% Results of formal reaction kinetics (to use the expression of~\cite{Aris1965,aris}) on the stationary point may offer solutions to problems in algebra, e.g. showing the existence of positive roots of a system of polynomial equations easier than by classical methods of algebraic geometry~\cite{dukaricerramijeralalebarromanovskitothweber,lichtblau}. In particular,  if the polynomial system is known to be the right hand side of an \ikde\ of a reversible or weakly reversible reaction network then the polynomial system is guaranteed to admit a positive root~\cite{borosPosSteadyStates}.
\item Lastly, results of formal reaction kinetics (to use an expression introduced in~\cite{Aris1965,aris}), e.g. on the existence of stationary points, may offer alternative methods for solving problems in algebraic geometry~\cite{dukaricerramijeralalebarromanovskitothweber,lichtblau}. For instance, one might be able to show the existence of positive roots of a polynomial if the system of polynomial equations is known to be the right hand side of the \ikde\ of a reversible or weakly reversible reaction network~\cite{borosPosSteadyStates}.

\end{enumerate}

\bigskip

The structure of our paper is as follows.
Section 2 introduces the essential concepts of reaction networks and mass action systems.
Section 3 formulates the problem  we are interested in, that of realizability of kinetic differential equations. 
Section 4 treats two special cases: finding realizations for \emph{compartmental models} (defined later) and weakly reversible networks. 
Section 5 focuses on the general problem of realizability. We first review existing algorithms available. Then we outline several procedures that modify a reaction network while preserving the system of differential equations, including adding and removing vertices from the reaction graph. 
Section 6 explores the relation between weakly reversible and complex balanced realizations. Here we work with families of symbolic kinetic differential equations.  
In Section 7, we prove the geometric meaning of zero deficiency, and discuss when is a realization unique. 
We also pose several conjectures that may be of interest to the reader. 
An Appendix contains a very large number of enlightening examples. 
The present paper intends to be a review of the realizability problem, while  containing some new results.

\bigskip

A few remarks on notation and the use of words are in order.
\begin{enumerate}
\item 
The set of vectors in $\rr^M$ with strictly positive coordinates is denoted by $\rr^M_{>0}$. Also $$\xb^{\yb} := x_1^{y_1}x_2^{y_2}\cdots x_M^{y_M}$$ for any $x_i > 0$ and $y_i \geq 0$. If $\Yb = (\yb_1, \yb_2, \ldots, \yb_N)$, then $$\xb^{\Yb} = (\xb^{\yb_1}, \xb^{\yb_2},\ldots, \xb^{\yb_N})^T.$$ 
\item
We use differential equation to mean a  \emph{system of ordinary differential equations.} In the same vein, we use \emph{polynomial} to usually mean a system of polynomial equations, i.e. a vector-valued function on $
\rr_{>0}^M$. By \emph{monomial} we mean a  scalar-valued function on $
\rr_{>0}^M$.
\item
We say a reaction network  \emph{induces} a \de, if we formulate the mass action type differential equation with given reaction rate coefficients, and a \de\  \emph{is realized} by a reaction network, if the network induces the given differential equation.
\end{enumerate}

%%%%%%%%%%%%%%%%%%%%%%%%%%%%%%%%%%%%%%%%%%%%%%%%%%%%%%%%%%%%%%%%%%%%%%%%%%%%%%%%%%%%%%%%%%%%%%%%%%%%%%%%%%%%%%%%%%%%%%%%%%%%%%%%%%%%%%%%%%%%%%%%%%%%%%%%%%%%%%%%%%%%%%%%%%%%%%%%%%%%%%%%%%%%%%%%  

\section{Chemical reaction networks and mass action systems}

Here we recapitulate some concepts of the mathematical description of reaction networks; for more details, see e.g.~\cite{feinbergbook,CraciunYu_Review, tothnagypapp}.

\label{def:basics}
A \df{reaction network} consists of a set of $R$ reaction steps
among the chemical species $\ce{X(1)}$, $ \ce{X(2)}, \ldots,  \ce{X(M)}$: 
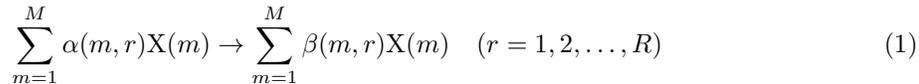
\begin{equation}
\sum_{m=1}^M\alpha(m,r) \ce{X($m$)} \rightarrow\sum_{m=1}^M\beta(m,r) \ce{X($m$)}
\quad (r=1,2,\dots,R)
\label{eq:reaction}
\end{equation}
with nonnegative integer \df{stoichiometric coefficients} $\alpha(m,r)$,  $\beta(m,r)$.
Its induced kinetic differential equation assuming mass action type kinetics
(and disregarding the change of temperature, pressure and reaction volume) is
\begin{equation}\label{eq:ikdecoord}
\frac{dx_m}{dt}=\sum_{r=1}^R(\beta(m,r)-\alpha(m,r))k_r \prod_{p=1}^Mx_p^{\alpha(p,r)}
\quad (m=1,2,\dots,M).
\end{equation}
The positive number $k_r$ is the \df{rate coefficient}%%
\footnote{Instead of \emph{rate constants} we use the term \emph{rate coefficients}, because they do depend on everything (pressure, temperature, volume), except concentrations, although we know that there is a heated argument about this question amongst reaction kineticists, see~\cite[p. 115]{lentedet}.} %%
for the $r$-th reaction.

\bigskip

On the two sides of arrows in \eqref{eq:reaction}, which represent a reaction step, one has formal linear combinations of the species, called \df{complexes}, the left one being the \df{reactant complex}, and the right one being the \df{product complex}. 
Associated to each complex is the vector of the coefficients of these linear combinations. By an abuse of notation, we will refer to this vector as a complex. 
The difference between the product complex and the reactant complex is the \df{reaction vector}. These vectors form the columns of the \df{stoichiometric matrix}. % $\gammab:=\betab-\alphab \in \zz^{M \times R}$, where
% \begin{align*}
% \betab := \begin{pmatrix}
%         \beta(1,1) & \beta(1,2) & \cdots & \beta(1, R) \\
%         \beta(2,1)  & \ddots & & \vdots \\
%         \vdots  && \ddots & \vdots \\
%         \beta(M, 1) & \cdots &\cdots & \beta(M, R) 
%     \end{pmatrix}
%     \quad \text{and} \quad
% \alphab := \begin{pmatrix}
%         \alpha(1,1) & \alpha(1,2) & \cdots & \alpha(1, R) \\
%         \alpha(2,1)  & \ddots & & \vdots \\
%         \vdots  && \ddots & \vdots \\
%         \alpha(M, 1) & \cdots &\cdots & \alpha(M, R) 
%     \end{pmatrix}
%     .
% \end{align*}
The \df{stoichiometric space} $\SSS$ is the linear span of the reaction vectors, where $S = \dim \SSS$. (This number can also be interpreted as the number of independent reaction steps.) The \df{stoichiometric compatibility class} of $\xb_0 \in \mathbb{R}_{>0}^M$ is the forward-invariant set $(\xb_0 + \SSS)_> := (\xb_0 + \SSS)\cap \mathbb{R}_{>0}^M$. The number of complexes is denoted by $N$.

With sets of reactions and complexes as described above, each reaction network is naturally associated to a directed graph.

\begin{definition}
The \df{reaction graph} (or \df{Feinberg-Horn-Jackson graph}) of a given reaction network is a finite directed graph with no self-loops, where the vertices are the complexes of the reaction network, and the edges are precisely the reactions. 
\end{definition}

The connected components of a reaction graph are also called \df{linkage classes}, and $L$ denotes the number of linkage classes in a reaction network. 
The \df{deficiency} of the reaction network is $\delta:=N-L-S$, where $N$ is the number of complexes (i.e., number of nodes in the reaction graph), and $S = \dim \SSS$. 
The strong components of the reaction graph are also called \df{strong linkage classes}. An \df{ergodic component} (or \df{strong terminal class}) is a strong component with no reaction step such that the reactant complex is in the component while the product complex is outside of it, i.e., an ergodic component is a strongly connected component of the reaction graph. The number of ergodic components will be denoted by $T$.

% We shall be interested in some classes of reaction networks.
\begin{definition}
\label{def:classesofreactions}
A reaction network is \df{reversible} if its reaction graph as a relation is symmetric, i.e., $i \to j$ is a reaction if and only if $j \to i$ is a reaction. 
It is \df{\WR}, if the transitive closure of its reaction graph is symmetric;
or if all the reaction arrows are edges of a directed cycle in the reaction graph;
or all of its strong components are ergodic.\footnote{Weak reversibility implies $T=L$, but the converse is false~\cite{feinberghornkinetic}.} 
\end{definition}

\begin{definition}
A reversible reaction network (with rate coefficients $k_{r+}$, $k_{r-}$ for the $r$-th pair of reversible reactions) 
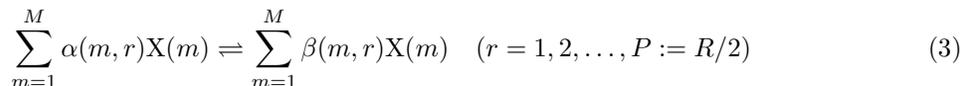
\begin{equation}\label{eq:revreacgeneral}
\sum_{m=1}^M\alpha(m,r)\ce{X($m$)}\rightleftharpoons\sum_{m=1}^M\beta(m,r)\ce{X($m$)}
\quad(r=1,2,\dots,P := R/2)
\end{equation}
is \df{detailed balanced} at the positive stationary concentration $\xb^*$ if for each pair of reversible reactions, the forward step proceeds at the same rate as the backward step; in other words, $\xb^*$ is detailed balanced if 
$$
    %  \frac{\mathrm{d}\xb}{\mathrm{d}t} = \vv 0 
    % \quad \text{implies} \quad 
    k_{r+}(\xb^*)^{\alphab_r} = k_{r-} (\xb^*)^{\betab_r}, 
    \quad 
    (r = 1,2,\ldots, P).
    % \gammab^\top\log(\cb_*)=\log(\kappab),
    \label{eq:forfredholm}
$$
\end{definition} 
% \noindent 
% \red{
% A detailed balanced stationary point exists if and only if for all $\ub\in\Ker(\gammab)$ one has $\kappab^\ub=1,$
% where $\kappab:=\red{\frac{\kb_+}{\kb_-}}$.}
% %%
% \blue{
%  -- \textbf{(Is this even true? $\gammab$ supposed to have reaction vectors as columns. We should remove?)}}

\bigskip 

The \ikde\ can be written in many different forms, see e.g.~\cite[Sec. 6.3]{tothnagypapp}.
We introduce the form that will be used frequently below.

For a reaction graph with $N$ complexes, let their associated vectors be columns of the matrix  $\Yb := (\yb_1,\yb_2,\dots,\yb_N)$. Implicitly we are imposing an ordering on the set of complexes. For a reaction $i \to j$ (from the $i$-th complex to the $j$th complex), assume that the rate coefficient is $k_{ji} > 0$. Define the matrix $\Kb$ to be
\begin{align}
    [\Kb]_{ij} &= \left\{\begin{array}{ll}
        k_{ij} & \text{ if } j \to i \text{ is a reaction in the graph} \\
        0 & \text{ otherwise }
    \end{array}
    \right.
\end{align}
so that $\Kb$ is a matrix with nonnegative entries, and $0$ along the diagonal. Then the column-conserving matrix $\Ab_k := \Kb-\diag(\Kb\T\oneb)$ is the \df{(weighted) Laplacian} of the reaction graph. A matrix with nonnegative off-diagonal entries, and is column-conserving is also called \df{Kirchoff}.

With these matrices defined, the \ikde\ \eqref{eq:ikdecoord} of the reaction network 
can be rewritten in the following form:
\begin{equation}\label{eq:ikdevecFHJ}
\frac{d\xb}{dt}=\Yb \cdot \Ab_k \cdot \xb^{\Yb},
\end{equation}
where $\xb^{\Yb}$ is the vector of monomials $(\xb^{\yb_1},\xb^{\yb_2}, \ldots, \xb^{\yb_N})^T$. 

Note that the  differential equation (\ref{eq:ikdevecFHJ}) is completely determined by the complexes (columns of $\Yb$) and the connectivity and rate coefficients of the reaction graph  (given by $\Ab_k$). %Thus, we call $(\Yb, \Ab_k)$ a \emph{realization} of the differential equation (\ref{eq:ikdevecFHJ}). 

\begin{definition}
The reaction system \eqref{eq:ikdevecFHJ} is \df{complex balanced} at the positive stationary concentration $\xb^*$, if all the complexes are destroyed and formed with the same rate at this concentration:
$$
\sum_{q \neq n} k_{qn}(\xb^*)^{\yb_n}=\sum_{q \neq n} k_{nq}(\xb^*)^{\yb_q}
$$
for all $n = 1,2,\ldots, N$. Equivalently, $\xb^*$ is complex balanced if $\Ab_k (\xb^*)^{\Yb} = \vv 0$. 
\end{definition}

If one remains within the framework of mass action kinetics -- as we do here -- then the Deficiency Zero Theorem provides a necessary and sufficient structural condition for complex balancing~\cite{feinberg1972cb, horn1972, hornjackson}: 
\begin{theorem}
A mass action system is complex balanced for all choices of the rate coefficients if and only if it is \WR\ and its deficiency is zero.
\end{theorem}

One might notice that in the Deficiency Zero Theorem, it is the mass action system that is described as complex balanced, not one of its stationary concentrations. This is because complex balancing implies a host of good structural and dynamical properties, including:
\begin{enumerate}
\item 
    If one of the stationary concentrations is complex balanced, then all stationary concentrations of the system is  complex balanced.
\item 
    There is exactly one stationary concentration within every stoichiometric compatibility class.
\item 
    Let $\xb^*$ be any complex balanced stationary concentration. The function
    \begin{align*}
        L(\xb) := \sum_{m=1}^M(\ln x_m - \ln x_m^* -1)
    \end{align*}
defined on $\rr^n_{>0}$ is a Lyapunov function, with global minimum at $\xb = \xb^*$.
\item 
    Every positive stationary concentration is locally asymptotically stable within its stoichiometric compatibility class.
\item 
    The set of complex balanced stationary concentrations $Z_\kb$ satisfies the equation $\ln Z_\kb = \ln \xb^* + \SSS^\perp$. 
\end{enumerate}

%%%%%%%%%%%%%%%%%%%%%%%%%%%%%%%%%%%%%%%%%%%%%%%%%%%%%%%%%%%%%%%%%%%%%%%%%%%%%%%%%%%%%%%%%%%%%%%%%%%%%%%%%%%%%%%%%%%%%%%%%%%
\section{Formulation of problems} 

In defining the system of differential equations for a mass action system, note that given a set of reactions and rate coefficients (or equivalently, given a reaction graph and positive weights for each edge), it is straightforward to write down the differential equation. Unfortunately, given a differential equation (satisfying a mild sign condition introduced below), there is in general no unique reaction graph that is associated to it~\cite{craciunpantea2008}. %%

A natural question arises: given a differential equation, what are the possible reaction graphs that could generate it if we assume mass action kinetics. All the information about a reaction graph is contained inside $(\Yb, \Ab_k)$.

\begin{definition}
	A \df{realization} of a differential equation (\ref{eq:ikdevecFHJ}) is the pair $(\Yb, \Ab_k)$, where $\Yb \in \zz_{\geq 0}^{M \times N}$ and $\Ab_k$ is the weighted Laplacian of a reaction graph. Equivalently, a realization is a set of complexes and a Feinberg-Horn-Jackson reaction graph. Two different realizations $(\Yb, \Ab_k)$ and $(\Yb', \Ab_k')$ are said to be \df{dynamically equivalent} if they give rise to the same differential equation under mass action kinetics. 
\end{definition}
Dynamical equivalence is written as linear relations in~\cite{craciunjinyu}.  In other words, the question we are asking in this work is: what are the possible dynamically equivalent realizations $(\Yb, \Ab_k)$ of a given system of differential equations? 

%%%%%%%%%%%%%%%%%%%%%%%%%%%%%%%%%%%%%%%%%%%%%%%
\bigskip 

Before searching for dynamically equivalent realizations, we must first ask whether or not a given  differential equation can arise from mass action kinetics. In other words, is it even possible to find one realization? We  provide a necessary and sufficient condition for realizability in Lemma~\ref{lemma:hungarian} below.  

Assume that we are given a system of differential equations of the form 
% $$\Zb=\left(\begin{array}{cccc}
%       \zb_1 & \zb_2 & \cdots & \zb_N
%     \end{array}\right)
% \in\R^{M\times N};
% \Yb=\left(\begin{array}{cccc}
%       \yb_1 & \yb_2 & \cdots & \yb_N
%     \end{array}\right)\in\N_0^{M\times N}.$$
\begin{equation}\label{eq:polynomial}
\dot{\xb}=\Zb\cdot  \xb^\Yb
=\sum_{n=1}^{N}\zb_n\xb^{\yb_n}
% =\sum_{n=1}^N\sum_{m=1}^Mz^m_{n}\eb_m\xb^{\yb_n}
% =\sum_{m=1}^M \eb_m\sum_{n=1}^Nz^m_{n}\xb^{\yb_n}
=\left(\begin{array}{c}
        \sum_{n=1}^{N}z_n^1 \xb^{\yb_n}\\
        \vdots \\
        \sum_{n=1}^{N}z_n^m \xb^{\yb_n} \\
        \vdots \\
        \sum_{n=1}^{N}z_n^M \xb^{\yb_n}
      \end{array}\right),
\end{equation}
where $\Zb := (\zb_1, \zb_2,\ldots, \zb_N)$, and $\Yb := (\yb_1, \yb_2,\ldots, \yb_N)$, and $\xb^{\yb} := x_1^{\yb_1}x_2^{\yb_2}\cdots x_M^{\yb_M}$ for any $x_i > 0$. 

\begin{definition}
\label{def:kinetic}
We say that a system of differential equations (\ref{eq:polynomial}) is \df{kinetic} if 
\begin{equation}\label{eq:negativecrosseffect}
 z^m_{n}<0 \quad \text{implies} \quad  (\yb_n)^m=y^m_{n}>0.
\end{equation}
\end{definition}
A system being kinetic reflects the physical assumption that in order to lose a chemical species via a reaction, that species must participate as a reactant. For example, the system 
    \begin{align*}
        \dot{x} &= y - x
        \\
        \dot{y} &= x - xz -y 
        \\
        \dot{z} &= xy - z
    \end{align*}
is not kinetic because the term $-xz$ in $\dot{y}$ does not involve the variable $y$. 

It turns out that a differential equation can arise from mass action kinetics if and only if it is \df{kinetic} in the sense of Definition~\ref{def:kinetic}. This has been shown  in~\cite{harstoth} using a constructive method.
\begin{lemma}[Hungarian Lemma]
\label{lemma:hungarian}
A polynomial system of differential equations of the form \eqref{eq:polynomial} can be realized by a mass action system if and only if the system of differential equations is kinetic.
\end{lemma}
% This statement is sometimes~\cite{cardellitribastonetschaikowski} called the \emph{Hungarian Lemma}, and a polynomial differential equation with no negative cross effect (or the right hand side itself) is said to be \emph{Hungarian}.

One realization (the \emph{canonical realization}) is the collection of reactions
\begin{equation}\label{eq:canonical}
\yb_n \xrightarrow{|z^m_n|} \yb_n+\sign(z^m_{n})\eb_m \quad(m=1,2,\ldots,M;n=1,2,\ldots,N),
\end{equation}
where $\eb_m$ is the $m$-th vector of the standard basis of $\mathbb{R}^M$.
Here the product complex vector has nonnegative integer coordinates because of the assumption in  \eqref{eq:negativecrosseffect}.

\begin{example}
Given the \de\
\begin{equation}\label{eq:logistic}
\dot{x}=k_1x-k_2x^2
\end{equation}  (with $k_1, \, k_2>0$),  it is possible to find an inducing reaction network with only two complexes defined  by the exponents of the monomials present in the right hand side:
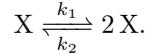
\begin{align*}
\ce{X <=>[$k_1$][$k_2$] 2X}.
\end{align*}
The right hand side contains two monomials, and there exists a realization with two complexes, and is also reversible, and has  deficiency zero. This also happens to be the canonical realization.
\end{example}
However, in general, the canonical realization is  quite complicated, contains too many complexes and reactions (see Supplementary Example \ref{ex:canonical}). Its main advantage is that it can be automatically constructed and used as a starting point for finding different realizations.

\bigskip

Now we present the set of problems we are interested in (also posed earlier in~\cite[Section 4.7]{erditoth}). \\

\noindent 
{\bf List of Problems.}\label{pr:enum}
\vspace{-0.5cm}
\begin{enumerate}
\item[\nonumber]
\item
Given a kinetic differential equation $\dot{\vv x} = \vv Z \cdot   \vv x^{\vv Y}$, find $\Ab_k$ such that $\Yb\cdot   \Ab_k = \Zb$.\label{pr:basic}
\item
Find reversible or weakly reversible; complex balanced or detailed balanced; or mass conserving realizations.
Decide if such a realization exists.
\label{pr:wr}
\item
Find realizations with minimal deficiency, with minimal number of reactions, or with minimal number of complexes.
\item
Without further restrictions, the question about uniqueness of $\Ab_k$ arises at all levels.
\label{pr:optimal}
\item
In all of the above problems, we may add more columns to $\Yb$~\cite{szederkenyiComment}. For example, if the first problem has no solution, we may consider the following: find $\overline{\Yb}$ whose first $N$ columns
form precisely $\Yb$ and find $\Ab_k$ such that $\overline{\Yb} \cdot  \Ab_k\cdot   \vv
x^{\overline{\Yb}} = \Zb \cdot  \vv x^{\Yb}$. In this case we can also ask what is the minimum number of columns that must be added.
\label{pr:enlarge}

\end{enumerate}

We will answer partially the questions above, especially pertaining to reversible, weakly reversible, complex balanced and detailed balanced realizations. We also consider uniqueness in Section~\ref{sec:uniqueness} and mass conserving realizations in the Supplementary.

\bigskip 

In considering the uniqueness question, the concepts of sparse and dense realizations have been useful~\cite{szederkenyi}. 
%%%%%%%%%%%%%%%%%%%%%%%%%%%%%%%%%%%%%%%
%\subsection{Sparse and dense realizations}
%%%%%%%%%%%%%%%%%%%%%%%%%%%%%%%%%%%%%%% 
The aim is to find an $M\times N$ matrix of nonnegative integer components $\Yb$ and a Kirchoff matrix $\Ab_k$ giving the product $\Zb=\Yb\cdot  \Ab_k$.
The pair $(\Yb, \Ab_k)$ is a  realization of the
matrix $\Zb$. Since $(\Yb, \Ab_k)$ uniquely identifies the reaction graph, when we refer to a subgraph of the realization, we mean a subgraph of the reaction graph corresponding to $(\Yb, \Ab_k)$.
\begin{definition}\label{def:sparsedense}
Suppose the set of complexes is fixed, i.e., the matrix $\Yb$ is given. If the number of zeros in the above setting in $\Ab_k$ is maximal, then a \df{sparse realization} has been obtained. If the number of zeros in $\Ab_k$ is minimal then a \df{dense realization} has been obtained.
\end{definition}
 Suppose $\Zb$ is given and fix the matrix $\Yb$, i.e., fix the set of complexes (both the reactant and product complexes) from which we will search for realizations. The following properties of dense and sparse realizations of a given chemical reaction network were proved in~\cite{szederkenyihangospeni}:
\begin{theorem}\label{th:realizations}
\begin{enumerate}
\item[\nonumber]
\item
The reaction graph of the dense realizations is unique.
\item
The reaction graph of any dynamically equivalent realization is a subgraph
of the dense realization.
\item
The reaction graph of a system of differential equations is unique if and only if the reaction graph of the dense and sparse realizations are identical.
\end{enumerate}
\end{theorem}

The above statements hold when the word ``realization'' is replaced by \emph{weakly reversible realization}~\cite{Acs2015}. 
The above results were extended to the case of \emph{constrained} realizations, where the rate coefficients fulfil arbitrary polytopic constraints~\cite{Acs2018}. An important special case of this is when a subset of possible reactions is excluded from the network. 

We have seen that the canonical representation may be the sparse realization.
% In some cases, based on Theorem \ref{thm:general} below 
% %%
% 	\blue{\textbf{-- (?? wait what does Theorem \ref{thm:general} have to do with sparse realizations?) }}
% %%
%  we shall be able to find the dense realization.
 Later, we shall give a complete solution to a generalization of part \ref{pr:basic} of Problem \ref{pr:enum} for the case when the numerical values
of the components of the matrices $\Zb$ and $\Yb$ are known.
We also review results for the solution of the other related problems in Section~\ref{sec:algorithms}.

To date, the problem of finding a realization has used tools and techniques from different mathematical disciplines, ranging from graph theory, combinatorics, linear algebra, to optimization.

%%%%%%%%%%%%%%%%%%%%%%%%%%%%%%%%%%%%%%%%%%%%%%%%%%%%%%%%%%%%%%%%%
\section{Realizations of compartmental models and reversible systems}
%%%%%%%%%%%%%%%%%%%%%%%%%%%%%%%%%%%%%%%%%%%%%%%%%%%%%%%%%%%%%%%%%

In this section, we consider two classes of examples: compartmental models and weakly reversible networks. Because of the simpler mathematical structure of compartmental models, the realization problem can be solved even if the rate coefficients are unspecified parameters, i.e., symbolic. Then we will look at an \emph{ad hoc} method to find reversible realizations.

\subsection{Compartmental Models}

A \df{compartmental system} is built using a subset of reactions of the form
\begin{align*}
\ce{X($m$) -> X($p$)},\quad\ce{X($n$) -> 0},\quad\ce{0 -> X($q$)}.
\end{align*}
It is \df{closed} if it only contains reactions of the first type,
it is \df{half-open} if reactions of the second type are also allowed,
and it is \df{open} if reactions of the third type may occur.
A \df{generalized compartmental system} only consists of the reactions
\begin{align*}
\ce{$y_m$X($m$) -> $y_p$X($p$)},\quad\ce{$y_n$X($n$) -> 0},\quad\ce{0 -> $y_q$X($q$)}
\end{align*}
with fixed positive integer numbers $y_1, y_2, \ldots, y_N$.
(It is understood that all the complexes contain a single species, and every species appear in a single complex.)
It is \df{closed}, if it only contains reactions of the first type,
it is \df{half-open} if reactions of the second type are also allowed,
and it is \df{open}, if reactions of the third type also may occur.

\bigskip

First we give a complete, although ineffective solution of the
first task of Problem \ref{pr:enum} in a slightly modified form.
Modification means that the empty complex will be dealt with separately,
because in some calculations this implies simplification.
First, we need two definitions.
\begin{definition}\cite[pp. 142--143]{szederkenyimagyarhangos}
A matrix $\Bb\in\R^{N\times N}$ is a \df{compartmental matrix}, if its off-diagonal elements are nonnegative, while its diagonal elements are smaller than or equal to the negative of the corresponding column sums: $b_{nn}\le-\sum_{p=1}^Nb_{pn}$. It is a \df{Kirchhoff matrix} if it is a compartmental matrix and its diagonal elements are equal to the negative of the corresponding column sums: $b_{nn}=-\sum_{ p\neq n}b_{pn}$.
\end{definition}
\noindent 
For example, the weighted Laplacian matrix of a reaction graph is Kirchoff. 

% \begin{definition}\cite[pp. 142--143]{szederkenyimagyarhangos}
%   $\Ab\in\R^{N\times N}$ is said to be a \emph{compartmental matrix}, if
% \begin{enumerate}
%   \item $a_{np}\ge0$, and
%   \item $a_{nn}\le-\sum_{p=1}^Na_{np}$.
% \end{enumerate}
% If instead of the second condition the stronger condition $a_{nn}=-\sum_{p=1}^Na_{np}$. holds, then $\Ab$ is a \emph{Kirchhoff matrix}.
% \end{definition}

Let $M,N\in\N;$ and suppose the vectors $\yb_1,\yb_2,\dots,\yb_{N}\in \zz_{\geq 0}^M$
are arbitrary nonzero vectors.
Let us write the reaction network we are interested in, in the form of
\begin{align}\label{eq:general1}
&\sum_{m=1}^{M}y_q^m\mbox{X}(m)\rightarrow\sum_{m=1}^{M}y_n^m\mbox{X}(m)\quad(n,q=1,2,\dots,N;n\neq q),\\
&\sum_{m=1}^{M}y_q^m\mbox{X}(m)\rightarrow0\quad(q=1,2,\dots,N),\label{eq:general2}\\
&0\rightarrow\sum_{m=1}^{M}y_n^m\mbox{X}(m)\quad(n=1,2,\dots,N).\label{eq:general3}
\end{align}
Let us introduce the notation (slightly different from the usual)
$$
\Kb:=\left(\begin{array}{c}
            k_{nq}
          \end{array}\right)_{n\neq q=1}^N,
\quad
\ub:=\left(\begin{array}{c}
            k_{01}\\
            k_{02}\\
            \vdots\\
            k_{0N}
          \end{array}\right),
\quad
\vb:=\left(\begin{array}{c}
            k_{10}\\
            k_{20}\\
            \vdots\\
            k_{N0}
          \end{array}\right),
$$
where $\Ab_k := \Kb - \diag(\Kb^T \oneb)$ is the weighted Laplacian of the reaction network. 
Note that  $k_{nq} \geq 0$.  Then, the induced kinetic differential equation of the reaction network
\eqref{eq:general1}--\eqref{eq:general3}
is
\begin{equation}
\dot{\cb}=\Yb(\Kb-\diag(\Kb\T\oneb-\ub))\cb^{\Yb}+\Yb\vb
\quad\mbox{ with }(\Yb=(\yb_1,\yb_2,\dots,\yb_N)).\label{eq:ikdegen}
\end{equation}
Upon denoting $\Zb:=\Yb(\Kb-\diag(\Kb\T\oneb-\ub)), \bb:=\Yb\vb$ one can say that
\eqref{eq:ikdegen} is of the form
\begin{equation}\label{eq:inducedgen}
\dot{\cb}=\Zb\cb^{\Yb}+\bb
\end{equation}
so that there exists a compartmental matrix $\Bb\in\R^{N\times N}$
and a nonnegative vector $\vb\in \rr_{\geq 0}^M$ so that
\begin{equation}
\Zb=\Yb\Bb,\quad\bb=\Yb\vb\label{eq:condition}
\end{equation}
hold. If $\ub=\nulb$, then $\Bb$ is a Kirchhoff matrix.
(One may call \eqref{eq:general2} \emph{outflow},
\eqref{eq:general3} \emph{inflow}, and they are missing if $\ub=\nulb$ or $\vb=\nulb$, respectively.)

A partial answer to the realization problem follows.
\begin{theorem}\label{thm:general}
Suppose that the differential equation
\begin{equation}\label{eq:tobeinduced}
\dot{\xb}=\Zb\xb^{\Yb}+\bb
\end{equation}
where
$\Yb=(\yb_1,\yb_2,\dots,\yb_N),\Zb\in\R^{M\times N}$ (with $M,N\in\N$)
has the property that there exist a compartmental matrix $\Bb$
and  a nonnegative vector $\vb\in \rr_{\geq 0}^M$ so that \eqref{eq:condition} holds. Then the equation has a realization of the form
\eqref{eq:general1}--\eqref{eq:general3}.
\end{theorem}
\begin{proof}
First we prove that the conditions of the theorem imply the differential equation is kinetic.
Suppose $Z^m_n<0$ for some indices $m,n$.
Then,
$$
Z^m_n=\sum_{q=1}^Ny_q^mB_n^q=\sum_{q\neq n}^Ny_q^mB_n^q+y_n^mB_n^n,
$$
and the last sum being nonnegative, $y_n^mB_n^n$ should be negative, thus $y_n^m>0$.

The components of $\Yb$ and $\vb$ are nonnegative, thus the components of $\bb=\Yb\vb$ are nonnegative, as well.

Next, we construct the inducing realization in the following way.
Let the reaction rate coefficients of the reaction network \eqref{eq:general1}--\eqref{eq:general3}  be defined as follows:
\begin{align*}
&k_{qn}:=B^q_{n}\quad(n,q=1,2,\dots,N;n\neq q),\\
&k_{n0}:=v_n\quad(n=1,2,\dots,N),\\
&k_{0n}:=-B^n_{n}-\sum_{q=1}^NB^q_{n}\quad(n=1,2,\dots,N).
\end{align*}
Then simple verification proves the statement.
\end{proof}
It may happen that a given kinetic differential equation has zero, one or an infinite number of realizations with complexes, or columns of $\Yb$, determined by the exponents of its monomials, see Supplementary Examples
\ref{ex:fourcomplexes}, \ref{ex:yinverse} and \ref{ex:opengcc} and Example \ref{ex:ThreeD}, respectively.
%%%%%%%%%%%%%%%%%%%%%%%%%%%%%%%%%%%%%%%%%%%%%%%%%%%%%%%%%%%%%%
%\subsection{Zero deficiency realizations}
\bigskip

%%%%%%%%%%%%%%%%%%%%%%%%%%%%%%%%%%%%%%%%%%%%%%%%%%%%%%%%%%%%%%
For some nonlinear \de s one can provide a zero deficiency realization.
Additional conditions may allow to have a realization which is also reversible or weakly reversible. One of the early approaches to this problem was~\cite[Theorem 9]{tothsztaki}, see also~\cite[Theorem 11.12]{tothnagypapp}  providing a necessary and sufficient condition that a generalized compartmental system induces a given differential equation. If the coefficients of the kinetic polynomial ODEs are known, then a deficiency zero realization can be determined via mixed integer linear programming, if it exists~\cite{Liptak2015}. Moreover, weakly reversible deficiency zero realizations can be determined using simple linear programming~\cite{szederkenyihangos}.
\begin{theorem}\label{th:gencomp}
Let $M\in\N; \Zb\in\R^{M\times M}, \Yb\in\N^{M\times M}, \bb\in \rr_{\geq 0}^M$.
The differential equation
\begin{equation}\label{eq:gencomp}
  \dot{\xb}=\Zb\xb^{\Yb}+\bb
\end{equation}
is the \ikde\ of a
\begin{enumerate}
\item
closed,
\item
half open,
\item
open
\end{enumerate}
\emph{generalized compartmental system} if and only if the following relations hold:
all the columns of $\Yb$ are positive multiples of different elements of the standard basis:
$$
\Yb=
\begin{bmatrix}
         y_1\eb_1 & y_2\eb_2 & \dots & y_M\eb_M
\end{bmatrix}; y_1,y_2,\dots,y_M\in\N;
$$
and
\begin{enumerate}
\item
$b_m=0; -z^m_{m},z^m_{p},\beta_m:=\sum_{p=1}^M\frac{z^p_{m}}{y_p} \geq 0; z^m_{m}=-\beta_my_m;$
\item
$b_m=0; -z^m_{m},z^m_{p},\beta_m \geq 0; z^m_{m}\le-\beta_my_m, \exists m:z^m_{m}<-\beta_my_m;$
\item
$-z^m_{m},z^m_{p},\beta_m \geq 0; z^m_{m}\le-\beta_my_m, \exists m:b_{m}>0;$
\end{enumerate}
respectively, where throughout $m,p\in\{1,2,\dots,M\}, m\neq p$.
\end{theorem}
\begin{proof}
Necessity is obvious. Instead of direct calculations
sufficiency can be learned from the application of Theorem \ref{thm:general}, as in this case $\Yb$ -- being diagonal with positive components -- is invertible, and its inverse can easily be calculated.
\end{proof}
\begin{remark}
\begin{enumerate}
\item[\nonumber]
\item
A necessary and sufficient condition of reversibility in the closed case is sign symmetry of the coefficient matrix. Similar conditions can be formulated for the other two cases, as well.
\item
Once we have a zero deficiency reversible realization, it is easy to check detailed balancing: only the circuit conditions~\cite{feinbergdb} are to be checked.
\end{enumerate}
\end{remark}
%%%%%%%%%%%%%%%%%%%%%%%%%%%%%%%%%%%%%%%%%%%%%%%%%%%%%%%%%%%%%%%%%%%%%%%
%\subsection{A method to find reversible realizations}
\bigskip 

\subsection{Finding reversible realizations}
%%%%%%%%%%%%%%%%%%%%%%%%%%%%%%%%%%%%%%%%%%%%%%%%%%%%%%%%%%%%%%%%%%%%%%%
Let us study reversibility, another relevant property, in more detail,
applying different methods---without taking care of the value of deficiency.

An equivalent and useful characterization of the \ikde\ of reversible reaction networks seems to be missing.
What we only have is a heuristic algorithm without the proof that it leads to a reversible realization if there exists one, and the realization is independent from the order of the steps.
\begin{example}\label{ex:ThreeD}
\begin{enumerate}
\item[\nonumber]
\item
Consider the polynomial equation
\begin{align}\label{eq:ThreeD}
\dot{x}&=b + s - g x + s z - g x z\\
\dot{y}&=b + s + s x - g y - g x y\\
\dot{z}&=b + s + s y - g z - g y z\label{eq:ThreeD3}
\end{align}
with $b,g,s>0$ (a transform of the first repressilator model in~\cite{dukaricerramijeralalebarromanovskitothweber}).
The canonical representation of this equation is shown in Figure~\ref{fig:ThreeD}. It is a reversible reaction network with exactly those seven complexes which are given as exponents of the monomials. It consists of a single linkage class and has a three-dimensional stoichiometric space, and it is also a sparse realization.
Reversibility allows us to deduce the existence of a positive stationary point by the theorems of~\cite{borosPosSteadyStates,orlovrozonoer2}.
In the third part below a more systematic approach to this example will show
an infinite number of realizations (not necessarily reversible), actually we shall receive the dense realization.
\begin{figure}[!ht]
  \centering
  \begin{tikzcd}[ampersand replacement=\&,column sep=small]
    \cf{X+Z} \arrow[rd,rightharpoonup,yshift=2pt,"g"] \& \& \& \& \cf{X+Y} \arrow[ld,rightharpoonup,yshift=-2pt,"g"] \\
    \& \cf{Z} \arrow[rd,rightharpoonup,yshift=2pt,"g",near start] \arrow[lu,rightharpoonup,yshift=-2pt,"s"] \& \& \cf{X} \arrow[ld,rightharpoonup,yshift=-2pt,"g"] \arrow[ru,rightharpoonup,yshift=2pt,"s"] \& \\
    \& \& \cf{0} \arrow[lu,rightharpoonup,yshift=-2pt,"b+s"] \arrow[ru,rightharpoonup,yshift=2pt,"b+s"] \arrow[d,rightharpoonup,xshift=1.5pt,"b+s"] \& \& \\
    \& \& \cf{Y} \arrow[d,rightharpoonup,xshift=1.5pt,"s"] \arrow[u,rightharpoonup,xshift=-1.5pt,"g"] \& \& \\
    \& \& \cf{Y+Z} \arrow[u,rightharpoonup,xshift=-1.5pt,"g"] \& \&
  \end{tikzcd}
  \caption{A reversible single linkage class reaction network of deficiency 3 to induce Eq.~\eqref{eq:ThreeD}-\eqref{eq:ThreeD3}. This realization is sparse as opposed to the dense realization shown in Figure~\ref{eq:dense}.}\label{fig:ThreeD}
\end{figure}
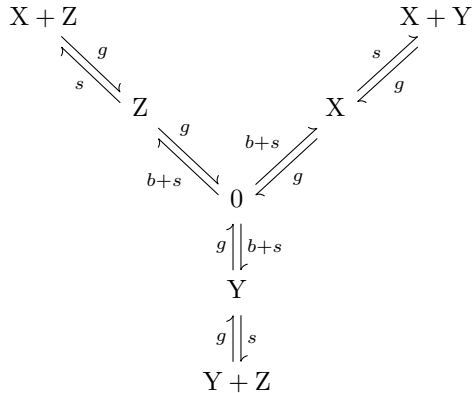
\item
To find a reversible realization one can also use a heuristic:
one tries to collect pairs of reactant and product complex vectors.
Suppose our goal is to find a reversible reaction network inducing \eqref{eq:ThreeD}--\eqref{eq:ThreeD3}.
Let us rewrite it in the following form
\begin{equation}\label{eq:rewrite301}
\left(\begin{array}{c}
  b + s - g x + s z - g x z \\
  b + s + s x - g y - g x y \\
  b + s + s y - g z - g y z
\end{array}\right)=
gxz\left(\begin{array}{c}
  -1 \\
  0 \\
  0
\end{array}\right)+\ldots.
\end{equation}
Now we consider the exponent of the monomial as a reactant vector $\alphab$ and
the vector with which this monomial has been multiplied, as the reaction vector $\betab-\alphab$, then one should have
$
\betab=
\left(\begin{array}{c}  -1 \\  0 \\  0\end{array}\right)+
\left(\begin{array}{c}  1 \\  0 \\  1\end{array}\right)=
\left(\begin{array}{c}  0 \\  0 \\  1\end{array}\right),
$
thus we need a term where the exponent of the monomial is $\betab$,
and it is multiplied by a positive constant and $\alphab-\betab$.
Thus we get:
\begin{equation}\label{eq:rewrite302}
\left(\begin{array}{c}
  b + s - g x + s z - g x z \\
  b + s + s x - g y - g x y \\
  b + s + s y - g z - g y z
\end{array}\right)=
gxz\left(\begin{array}{c}  -1 \\  0 \\  0\end{array}\right)+
sz\left(\begin{array}{c}  1 \\  0 \\  0\end{array}\right)+
\ldots.
\end{equation}
Repeating steps of this kind we arrive at the decomposition
\begin{align*}
gxz&\left(\begin{array}{c}  -1 \\  0 \\  0\end{array}\right)+
sz\left(\begin{array}{c}   1 \\  0 \\  0\end{array}\right)+
gyz\left(\begin{array}{c}  0 \\  0 \\  -1\end{array}\right)+
sy\left(\begin{array}{c}  0 \\  0 \\  1\end{array}\right)+\\
gxy&\left(\begin{array}{c}  0 \\  -1 \\  0\end{array}\right)+
sx\left(\begin{array}{c}  0 \\  1 \\  0\end{array}\right)+
gx\left(\begin{array}{c} -1 \\  0 \\  0\end{array}\right)+
(b+s)\left(\begin{array}{c}  1 \\  0 \\ 0\end{array}\right)+\\
gy&\left(\begin{array}{c} 0 \\ -1 \\  0\end{array}\right)+
(b+s)\left(\begin{array}{c}  0 \\  1 \\ 0\end{array}\right)+
gz\left(\begin{array}{c} 0 \\  0 \\  -1\end{array}\right)+
(b+s)\left(\begin{array}{c}  0 \\  0 \\ 1\end{array}\right)
\end{align*}
which again leads to the network in Figure~\ref{fig:ThreeD}.
\begin{conjecture}
If there exists a reversible realization then the above process (independently from the order of actions) ends and necessarily provides one of them which may or may not depend on the order of steps.
\end{conjecture}
\item
Finally, one can apply Theorem \ref{thm:general} to get all the realizations (with the exponents as complexes only) and choose from them the reversible one(s).
Suppose again that the Eqs. \eqref{eq:ThreeD}--\eqref{eq:ThreeD3} are given, and our goal is to find a reaction network inducing it using Theorem \ref{thm:general}.
In this case we have
\begin{equation}
\Zb=\left(\begin{array}{rrrrrr}
             -g & 0 & s & -g & 0 & 0 \\
             s & -g & 0 & 0 & -g & 0 \\
             0 & s & -g & 0 & 0 & -g
          \end{array}\right),\quad
\bb=\left(\begin{array}{c}
             b+s \\
             b+s \\
             b+s
          \end{array}\right),\quad
\Yb=\left(\begin{array}{cccccc}
            1 & 0 & 0 & 1 & 1 & 0 \\
            0 & 1 & 0 & 0 & 1 & 1 \\
            0 & 0 & 1 & 1 & 0 & 1
          \end{array}\right)
\end{equation}
and we have to find a compartmental solution in $\Bb$ of the equation $\Zb=\Yb\Bb$, or
$$
\begin{array}{rclrclrcl}
-g&=&-b_{21}-b_{31}-b_{61}-u_1 &s&=&b_{21}+b_{51}+b_{61} & 0&=&b_{31}+b_{41}+b_{61}\\
0&=&b_{12}+b_{42}+b_{52}  &-g&=&-b_{12}-b_{32}-b_{42}-u_2&
s&=&b_{32}+b_{42}+b_{62}\\
s&=&b_{13}+b_{43}+b_{53}  &0&=&b_{23}+b_{53}+b_{63} &
-g&=&-b_{13}-b_{23}-b_{53}-u_3\\
-g&=&-b_{24}-b_{34}-b_{64}-u_4 &0&=&b_{24}+b_{54}+b_{64} &
0&=&-b_{14}-b_{24}-b_{54}-u_4\\
0&=&-b_{25}-b_{35}-b_{65}-u_5  &-g&=&-b_{15}-b_{35}-b_{45}-u_5&
0&=&b_{35}+b_{45}+b_{65}\\
0&=&b_{16}+b_{46}+b_{56}  &0&=&-b_{16}-b_{36}-b_{46}-u_6 &
-g&=&-b_{16}-b_{26}-b_{56}-u_6.
\end{array}
$$
which for the coordinates has the following consequences.
$$\begin{array}{ccccccccc}
b_{31} = 0&b_{41} = 0&b_{61} = 0&
b_{12} = 0&b_{42} = 0&b_{52} = 0&
b_{23} = 0&b_{53} = 0&b_{63} = 0\\
b_{24} = 0&b_{54} = 0&b_{64} = 0&
b_{35} = 0&b_{45} = 0&b_{65} = 0&
b_{16} = 0&b_{46} = 0&b_{56} = 0
\end{array}
$$
and then the equations reduce to
$$\begin{array}{rcrrcrrcrrcr}
-g&=&-b_{21}-u_1   &
 s&=& b_{21}+b_{51}&
-g&=&-b_{32}-u_2   &
 s&=& b_{32}+b_{62}\\
 s&=& b_{13}+b_{43}&
-g&=&-b_{13}-u_3   &
-g&=&-b_{34}-u_4   &
 0&=&-b_{14}-u_4   \\
 0&=&-b_{25}-u_5   &
-g&=&-b_{15}-u_5   &
 0&=&-b_{36}-u_6   &
-g&=&-b_{26}-u_6
\end{array}
$$
having the solution
$$\begin{array}{rclrclrclrcl}
b_{21}&=&g-u_1  &
b_{51}&=&s-g+u_1&
b_{32}&=&g-u_2  &
b_{62}&=&s-g+u_2\\
b_{13}&=&g-u_3  &
b_{43}&=&s-g+u_3&
b_{34}&=&g-u_4  &
b_{14}&=&-u_4   \\
b_{25}&=&-u_5   &
b_{15}&=&g-u_5  &
b_{36}&=&-u_6   &
b_{26}&=&g-u_6
\end{array}
$$
implying that
$$
u_4=u_5=u_6=0.
$$
Now the kinetic matrix of the dense realization (without inflow) is
\begin{align*}\left(
\begin{array}{lllrrr}
-s-u_1    &0       &g-u_3   & 0& g& 0\\
 g-u_1    &-s-u_2  &0       & 0& 0& g\\
0         & g-u_2  &-s-u_3  & g& 0& 0\\
0         &0       &-g+s+u_3&-g& 0& 0\\
-g+s+u_1  &0       &0       & 0&-g& 0\\
0         &-g+s+u_2&0       & 0& 0&-g
\end{array}\right)
\end{align*}
under the conditions that
$$
s-g\le u_1\le g\quad
s-g\le u_2\le g\quad
s-g\le u_3\le g
$$
hold.
To have no indices let us introduce $\alpha:=u_1, \beta:=u_2, \gamma:=u_3$.
Now consider two cases:
\begin{enumerate}
\item
If $g\geq s$, take $\alpha$, $\beta$ and $\gamma$ such that $0\leq\alpha,\beta,\gamma\leq s$.
\item
If $s\geq g$, take $\alpha$, $\beta$ and $\gamma$ such that $s-g\leq\alpha,\beta,\gamma\leq s$.
\end{enumerate}
As to the inflow: the components of the solution $\vb$ to $\Yb\vb=\bb$ should fulfil
$$
v_1+v_4+v_5=b+s\quad
v_2+v_5+v_6=b+s\quad
v_3+v_4+v_6=b+s,
$$
which means that any $\vb$ defines a solution for which
$$
0\le v_4, 0\le v_5, 0\le v_6, v_4+v_5\le b+s,
v_5+v_6\le b+s, v_4+v_6\le b+s
$$
hold; and then
$v_1:=b+s-v_4-v_5, v_2:=b+s-v_5-v_6, v_3:=b+s-v_4-v_6$.

%KITÖRÖLNÉM

Summing up, Figure~\ref{fig:ThreeD} shows a sparse realization, while Figure~\ref{eq:dense} shows the dense realization. We also obtain 
the condition of reversibility:
$\alpha=\beta=\gamma=s$ and $v_4=v_5=v_6=0$,
and that it is exactly the sparse realization which is reversible.
\begin{figure}[!ht]
  \centering
  \begin{tikzcd}[ampersand replacement=\&,column sep=small]
    \cf{X+Y} \arrow[rd,rightharpoonup,yshift=2pt,"g"] \& \& \& \& \cf{X+Z} \arrow[ld,rightharpoonup,yshift=-2pt,"g"] \\
    \& \cf{X} \arrow[rd,rightharpoonup,yshift=2pt,"g-s+\alpha",near end] \arrow[lu,rightharpoonup,yshift=-2pt,"\alpha"] \arrow[ddr,bend right=20,"s-\alpha"'] \& \& \cf{Z} \arrow[ld,rightharpoonup,yshift=-2pt,"g-s+\gamma"] \arrow[ru,rightharpoonup,yshift=2pt,"\gamma"] \arrow[ll,bend right=20,"s-\gamma"'] \& \\
    \& \& \cf{0} \arrow[lu,rightharpoonup,yshift=-2pt,"b+s-v_4-v_5"] \arrow[ru,rightharpoonup,yshift=2pt,"b+s-v_5-v_6",near end] \arrow[d,rightharpoonup,xshift=1.5pt,"b+s-v_4-v_6"] \arrow[uull,bend right=30,"v_4"'] \arrow[dd,bend right=30,"v_6"'] \arrow[uurr,bend right=30,"v_5"'] \& \& \\
    \& \& \cf{Y} \arrow[d,rightharpoonup,xshift=1.5pt,"\beta"] \arrow[u,rightharpoonup,xshift=-1.5pt,"g-s+\beta"] \arrow[ruu,bend right=20,"s-\beta"'] \& \& \\
    \& \& \cf{Y+Z} \arrow[u,rightharpoonup,xshift=-1.5pt,"g"] \& \&
  \end{tikzcd}
\caption{A weakly reversible single linkage class realization of deficiency 3 to induce Eq.~\eqref{eq:ThreeD}-\eqref{eq:ThreeD3}. This realization is dense as opposed to the sparse realization shown in Figure~\ref{fig:ThreeD}.  }\label{eq:dense}
\end{figure}
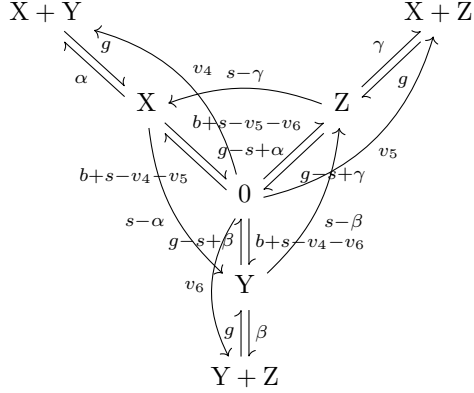
\end{enumerate}
\end{example}
In this example we were lucky to receive a canonical realization with so many nice properties. The reader can easily verify that this is the case with the second three dimensional repressilator model and with the six dimensional one of~\cite{dukaricerramijeralalebarromanovskitothweber}: the canonic realization is again reversible thus assures the existence of a positive stationary point.

Let us emphasize that via a known result in reaction kinetics we were able to assert the existence of a positive solution to a second degree polynomial (in three variables), or to put it another way, we could
\emph{apply formal reaction kinetics in mathematics}.

As to uniqueness of the stationary point: it does not follow from the mentioned theorems. One may try several methods, see e.g.~\cite{joshishiu}.
However, uniqueness of the reversible realization itself does follow (in this special case) from the method based on Theorem \ref{thm:general}.
%%%%%%%%%%%%%%%%%%%%%%%%%%%%%%%%%%%%%%%%%%%%%%%%%%%%%%%%%%%%%%%%%%%%%%%%%%%%%%%%%%%%%%%%%%%%%%%%%%%%%%%%%%%%%%%%%%%%%%%%%%%%%%%%%%%%%%%%%%%%%%
\section{On computing realizations of kinetic equations}

In the previous section, we looked at two classes of systems with certain network properties, namely that of (generalized) compartmental models (i.e., complexes consisting of a single species) and that of reversible networks. In this section, we will provide an algorithmic approach to the realizability problem when the coefficients in the ODE are given. To date, there are algorithms based on linear programming (LP) and mixed integer linear programming (MILP) techniques. These are reviewed in Section~\ref{sec:algorithms}.

In the previous section, we only use the complexes that show up as exponents of monomials in the kinetic differential equation. This may not necessarily be the case for more general examples~\cite{szederkenyiComment}. However, once \emph{a set of complexes has been decided on} (and must include the exponents of monomials), the algorithms presented below can be used to compute realizations with desirable structural properties.

In Section~\ref{sec:modnetworkDE}, we consider ways to modify an existing realization while preserving dynamical equivalence, including adding to/subtracting from the set of complexes. Finally, we consider cases when a realization with certain properties (e.g., complex balanced) is independent of the choice of the set of complexes.

%%%%%%%%%%%%%%%%%%%%%%%%%%%%%%%%%%%%%%%%%%%%%%%%%%%%%%%%%%%%%%%%
\subsection{Review of algorithms}
\label{sec:algorithms}

Qualitative properties can be rewritten in terms of a \emph{mixed integer linear programming} (MILP) problem~\cite{szederkenyihangospeni}.  %More specifically, they can compute realizations with given properties. 
The works~\cite{szederkenyihangostuza,johnstonsiegelszederkenyiwr} determine weakly reversible realizations, while~\cite{szederkenyihangos} provides reversible and detailed balanced realizations.
The paper~\cite{schusterschuster} finds detailed balanced subnetworks in reactions.

\bigskip 

%%%%%%%%%%%%%%%%%%%%%%%%%%%%%%%%%%%%%%%%%%%
Recall the discussion about possibly enlarging the set of complexes (Problem~\ref{pr:enlarge}) from $\Yb$ to $\overline{\Yb}$. We write the dynamical equivalence conditions for finding a realization using linear programming method with this predetermined complex set $\overline{\Yb}$. Note that the columns of $\Yb$ must also be columns of $\overline{\Yb}$. We want to find a realization of the kinetic differential equation $\dot{\xb} = \Zb \cdot \xb^{\Yb}$ with the complex set $\overline{\Yb}$. In other words, we want to find a  Kirchoff matrix $\Ab_k$ such that 
    \begin{align}
    \label{eq:realEQ1}
        \overline{\Yb} \cdot  \Ab_k \cdot  \xb^{\overline{\Yb}}= \Zb \cdot  \xb^{\Yb}.
    \end{align}
We may add columns of $\vv 0$ to the matrix $\Zb$ and obtain $\overline{\Zb}$, where the number of columns added is the number of additional columns in $\overline{\Yb}$ compared to $\Yb$. Then Eq.~\eqref{eq:realEQ1}  becomes
    \begin{align}
    \label{eq:realEQ2}
        \overline{\Yb} \cdot \Ab_k \cdot  \xb^{\overline{\Yb}}= \overline{\Zb} \cdot  \xb^{\overline{\Yb}}.
    \end{align}

% %%%%%%%%%%%%%%%%%%%%%%%%%%%%%%%%%%%%%%%%%%%
% Now we give the solution of a more general problem, than the original Problem \ref{pr:basic}. In Problem \ref{pr:basic}, we stated the problem of finding realizations of a kinetic differential equation
% $\dot{\xb} = \Zb \cdot \xb^{\Yb}$
% as searching for a Kirchhoff matrix $\Ab_k$ such that
% $\Yb \Ab_k = \Zb$.
% If there is a particular set of complexes (including the original ones) that we would like to use in some realization, we may add more columns to $\Yb$ forming the matrix $\overline{\Yb}$.
% We define a matrix $\overline{\Zb}$ by adding columns of 0.
% Thus,  Eq.~(\ref{eq:ikdevecFHJ}) can be rewritten (by adding zero columns to $\Zb$) as
% \begin{equation}
% \dot{\xb}=\overline{\Zb}\cdot \xb^{\overline{\Yb}}
% \end{equation}
% Given: a kinetic polynomial \de\
% \begin{equation}
% \dot{\xb}= \Zb \cdot \xb^\Yb \label{eq:kin1}
% \end{equation}
% where $\Zb$ contains the coefficients of the monomials and the columns of $\Yb$ are the exponents of the monomials;
% and also given a preferred complex set $\overline{\Yb}$ so that all the columns of $\Yb$ are columns of $\overline{\Yb}$\\
% (\ref{eq:kin1}) can be rewritten (by adding zero columns to $\Zb$) as
% \begin{equation}
% \dot{\xb}=\overline{\Zb}\cdot \xb^{\overline{\Yb}}
% \end{equation}
Therefore, kinetic realizability using $\overline{\Yb}$ is equivalent to the feasibility of the following convex constraint set
\begin{align}
\label{eq:RealizationGeneral-DE}
\overline{\Yb}\cdot \Ab_k & = \overline{\Zb}&\text{(dynamical equivalence)} \\
\label{eq:RealizationGeneral-kcnst}
[\Ab_k]_{ij}         & \ge 0~\text{for}~i\ne j&\text{(nonnegativity of rate coefficients)} \\
\label{eq:RealizationGeneral-colconserv}
\oneb\T \cdot \Ab_k  &= \nulb&\text{(column conservation in $\Ab_k$)}
%\Bb\cdot \overline{\Zb}   & \ge \nulb&\text{(lack of negative cross-effects)}.
\end{align}

\bigskip 

Next, we review algorithms for finding realizations with certain structural properties, namely, reversible, weakly reversible, or complex balancing realizations. In Eq.~\eqref{eq:RealizationGeneral-DE}-\eqref{eq:RealizationGeneral-colconserv} above, the matrices $\Yb$ and $\Zb$ do not show up explicitly. Hence, to simplify notation, in what follows below, we let $\Yb$ denote the complexes to be considered, some of which may not appear explicitly in the differential equation, and $\Zb$ defines a kinetic differential equation (with possibly columns of $\vv 0$ so that $\Zb$ is the same size as $\Yb$).

\bigskip 

To find a \textbf{reversible realization}, consider the following feasibility problem~\cite{Johnston2012a}:
\begin{align}
    \Yb \cdot \Ab_k &= \Zb \\
    [\Ab_k]_{ij} &\geq 0 \text{ for } i \neq j \\
    \oneb^T \cdot \Ab_k &= \nulb \\
    [\Ab_k]_{ij} \geq \epsilon 
 &\iff   [\Ab_k]_{ji} \geq \epsilon \text{ for } i \neq j. \label{eq:RealizationRev}
\end{align}
The statement (\ref{eq:RealizationRev}) is the condition that the  reaction step $i \rightarrow j$ is present if and only if reaction $j \rightarrow i$ is present. This feasibility problem can be solved in the framework of mixed integer linear programming (MILP). 

\bigskip 

To find a  \textbf{weakly reversible realization}, consider the following feasibility problem~\cite{Johnston2012}: 
\begin{align}
    \Yb \cdot \Ab_k &= \Zb \\
    [\Ab_k]_{ij} &\geq 0 \text{ for } i \neq j \\
    \oneb^T \cdot \Ab_k &= \nulb \\
\epsilon [\Ab_k]_{ij}&\leq [\Ab'_k]_{ij}, 
\, \epsilon [\Ab'_k]_{ij}\leq [\Ab_k]_{ij} 
\text{ for } i \neq j
    \label{eq:wrMILP}
\\
\Ab'_k \cdot \oneb &= \nulb.  \label{eq:wrRescaled}
\end{align}
The inequalities \eqref{eq:wrMILP} ensure that the structure of $\Ab_k$ and $\Ab_k'$ are the same, i.e., a reaction is present in one if and only if it is present in the other. Since every weakly reversible network is complex balanced for some choice of parameter, the algorithm searches for a complex balanced $\Ab_k'$ (in Eq.~\eqref{eq:wrRescaled}), which shares the network structure we want but there is no relations in terms of the rate coefficients we seek. 

In this MILP problem, the decision variables are the off-diagonal entries of $\Ab_k$, the off-diagonals entries of another Kirchoff matrix $\Ab'_k$, while
${\Zb, \Yb}$ are given matrices, and $\epsilon \ll 1$ is a fixed small constant. We will explore the relationship between finding weakly reversible and complex balanced realizations in Section~\ref{sec:WRandCB}. More efficient methods to compute weakly reversible realizations are available; for example, a polynomial time algorithm can be found in~\cite{rudanszederkenyihangospeni2014}.

%%%%%%%%%%%%%%%%%%%%%%%%%%%%%%%%

\bigskip 

To find a \textbf{detailed balanced realization}, supposed we have the numerical value of one stationary concentration $\xb^*$ and consider~\cite{szederkenyihangos}:
\begin{align}
    \Yb \cdot \Ab_k &= \Zb \\
    [\Ab_k]_{ij} &\geq 0 \text{ for } i \neq j \\
    \oneb^T \cdot \Ab_k &= \nulb \\
    [\Ab_k]_{ji} (\xb^*)^{\yb_i} &=  [\Ab_k]_{ij} (\xb^*)^{\yb_j} \text{ for each } i \neq j.
\end{align}

%%%%%%%%%%%%%%%%%%%%%%%%%%%%%%%%

\bigskip 

Finally, to find a \textbf{complex balanced realization}, supposed we have the numerical value of one stationary concentration $\xb^*$ and consider~\cite{szederkenyihangos}:
\begin{align}
    \Yb \cdot \Ab_k &= \Zb \\
    [\Ab_k]_{ij} &\geq 0 \text{ for } i \neq j \\
    \oneb^T \cdot \Ab_k &= \nulb \\
    - [\Ab_k]_{ii} (\xb^*)^{\yb_i} &= \sum_{j\neq i} [\Ab_k]_{{ij}} (\xb^*)^{\yb_j} \text{ for each } i.
\end{align}

% %%%%%%%%%%%%%%%%%%%%%%%%%%%%%%%%

% \bigskip 

% \remg{To wrap up our review, algorithms similar to the above exist for detailed balanced and complex balanced realizations~\cite{szederkenyihangos}.} The main difference between these two and the algorithms presented above is that a stationary concentration value is needed as input to the algorithms. 

%%%%%%%%%%%%%%%%%%%%%%%%%%%%%%%%%%%%%%%%%%%%%%%%%%%%%%%%%%%%%%%%%%%%%%%%%%%%%%%%%%%%%%%

\subsection{Modifying network while preserving dynamical equivalence}
\label{sec:modnetworkDE}

In this section, we introduce several modifications to an existing realization while preserving dynamical equivalence. These include:
\begin{enumerate}
    \item positive or convex combinations of realizations,
    \item 
    adding new reactions, and 
    \item 
    eliminating complexes.
\end{enumerate}
This list is non-exhaustive, but we hope that some of the ideas presented in this section will motivate other types of modifications that preserve dynamical equivalences. 

\bigskip 

\subsubsection{Positive and convex combinations of realizations}

Suppose we have found several realizations $(\Yb, \Ab_k^1), \, (\Yb, \Ab_k^2), \dots, (\Yb, \Ab_k^I)$ to the kinetic differential equation $\dot{\xb} = \Zb \cdot \xb^{\Yb}$. Then we can easily produce other realizations which may be simpler, or fulfil some further properties~\cite{Csercsik2012a}.

\begin{theorem}
Suppose that $\Ab^1_{k}$,\dots,$\Ab^I_{k}$ are realizations of a kinetic differential equation, with complex set $\Yb$,
and let $\Ab_{k}:=\sum_{i=1}^I\alpha_i\Ab^i_{k}$ for some $\alpha_i>0$, $i=1,\dots,I$.
\begin{enumerate}
\item
$\Ab_{k}$ is a Kirchhoff matrix, and realizes the equation $\dot{\xb} = (\sum_i \alpha_i \Yb \cdot  \Ab_k^i) \xb^{\Yb}$.
\item 
    If $\sum_{i=1}^I \alpha_i = 1$, then $(\Yb, \Ab_k)$ realizes the same kinetic differential equation $\dot{\xb} = \Zb \cdot \xb^{\Yb}$.
\item
    If for all $i=1,\dots,I$, the matrix $\Ab^i_{k}$ defines a reversible realization of some kinetic equation, then $\Ab_{k}$ defines a reversible realization of some kinetic equation.
\item
    If for all $i=1,\dots,I$, the matrix $\Ab^i_{k}$ defines a weakly reversible realization of some kinetic equation, then $\Ab_{k}$ defines a weakly reversible realization of some kinetic equation.
\end{enumerate}
\end{theorem}
\begin{proof}
\begin{enumerate}
\item[\nonumber]
\item Obviously, the off-diagonal terms of $\Ab_k$ are nonnegative. Also the column sums of $\Ab_k$ are conserved. Thus $\Ab_k$ is Kirchoff. We see that
    \begin{align*}
        \Yb \cdot \Ab_k \cdot \xb^{\Yb} &=  \sum_{i=1}^I \alpha_i \Yb \cdot  \Ab_k^i \cdot \xb^{\Yb}.
    \end{align*}
\item 
    This trivially follows from the calculation above.
\item For any $p \neq q$, if $[\Ab_k]_{pq} > 0$, then $[\Ab_k^i]_{pq} > 0$ for some $ 1 \leq i \leq I$. Since $\Ab_k^i$ defines a reversible realization, it must be the case that $[\Ab_k^i]_{qp} > 0$, hence $[\Ab_k]_{qp} > 0$, i.e., $\Ab_k$ defines a reversible realization.
\item
The proof is similar to the above. For any $p \neq q$, if $[\Ab_k]_{pq} > 0$, there must be some $[\Ab_k^i]_{pq} > 0$. Since $\Ab_k^i$ defines a weakly reversible realization, the reaction $q \to p$ must be part of a cycle, i.e., the off-diagonal terms in $\Ab_k^i$ corresponding to the edges in this cycle are positive, hence those positions in $\Ab_k$ are also positive. In other words, this particular cycle appears in the graph defined by $\Ab_k$ as well.
\end{enumerate}
\end{proof}
%%%%%%%%%%%%%%%%%%%%%%%%%%%%%%%%%%%%%%%%%%%%

\bigskip 

\subsubsection{Adding new reactions to a realization}

Now we look at  modifying an existing realization by adding new reactions.  Suppose that $\Ab_{k}$ defines a realization with complexes $\yb_1, \, \yb_2, \, \ldots, \, \yb_N$, and suppose that $\yb$ is a convex combination of the remaining complexes, i.e.,  $\yb=\sum_{n=1}^{N} v_n\yb_n$ where $ v_n\ge0$ for $n=1,2,\dots,N$ and $\sum_{n=1}^{N} v_n=1$.

Then we can add the following reactions: for every $n$ such that $v_n \neq 0$, add a reaction from $\yb$ to the complex $\yb_n$, with rate coefficient $v_n$. These reactions together contributes $\vv 0$ to the differential equation.

We wonder if we can add reactions in the opposite direction, i.e., reactions of the form $\yb_n \to \yb$. There are cases when it is possible and cases when it is not. While we do not have an answer to this question, we illustrate the complexity using two examples.

\begin{example}
First we show an example where adding a complex situated in the convex cone of the original complexes seems to be superfluous: its effect on the induced kinetic differential equation can be expressed only using reactions among the original complexes. Consider the system  
    
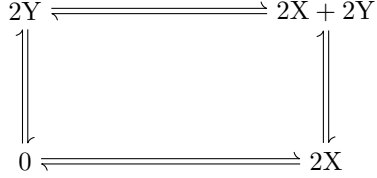
\begin{center}
    \begin{tikzpicture}
        \node (0) at (0,0) {0};
        \node (2x) at (4,0)  {\cf{2X}};
        \node (2x2y) at (4,2)  {\cf{2X + 2Y}};
        \node (2y) at (0,2) {\cf{2Y}};
        \draw [-{>[harpoon]}, transform canvas={yshift=1}] (0)--(2x);
        \draw [-{>[harpoon]}, transform canvas={yshift=-1}] (2x)--(0);
        \draw [-{>[harpoon]}, transform canvas={yshift=1}] (2y)--(2x2y);
        \draw [-{>[harpoon]}, transform canvas={yshift=-1}] (2x2y)--(2y);
        \draw [-{>[harpoon]}, transform canvas={xshift=-1}] (0)--(2y);
        \draw [-{>[harpoon]}, transform canvas={xshift=1}] (2y)--(0);
        \draw [-{>[harpoon]}, transform canvas={xshift=-1}] (2x)--(2x2y);
        \draw [-{>[harpoon]}, transform canvas={xshift=1}] (2x2y)--(2x);
        %%%%
    \end{tikzpicture}
    \end{center}
To add the reaction step $\cf{2X} \to  \cf{X+Y}$  %\cf{2X ->[$k$] X + Y}$ 
to get the network 
 
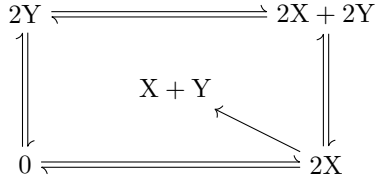
\begin{center}
    \begin{tikzpicture}
        \node (0) at (0,0) {0};
        \node (2x) at (4,0)  {\cf{2X}};
        \node (2x2y) at (4,2)  {\cf{2X + 2Y}};
        \node (2y) at (0,2) {\cf{2Y}};
        \node (xy) at (2,1) {\cf{X+Y}};
        \draw [-{>[harpoon]}, transform canvas={yshift=1}] (0)--(2x);
        \draw [-{>[harpoon]}, transform canvas={yshift=-1}] (2x)--(0);
        \draw [-{>[harpoon]}, transform canvas={yshift=1}] (2y)--(2x2y);
        \draw [-{>[harpoon]}, transform canvas={yshift=-1}] (2x2y)--(2y);
        \draw [-{>[harpoon]}, transform canvas={xshift=-1}] (0)--(2y);
        \draw [-{>[harpoon]}, transform canvas={xshift=1}] (2y)--(0);
        \draw [-{>[harpoon]}, transform canvas={xshift=-1}] (2x)--(2x2y);
        \draw [-{>[harpoon]}, transform canvas={xshift=1}] (2x2y)--(2x);
        %%%%
        \draw [->] (2x)--(xy);
    \end{tikzpicture}
    \end{center}
% then the induced kinetic differential equation of the system will be:
%     \begin{align*}
%         \dot{x} &= \cdots  - k x^2, \\
%         \dot{y} &= \cdots  + k x^2. 
%     \end{align*}
% But the same differential equation is induced if instead the reaction step \ce{2X -> 2Y}
% is added with the rate coefficient $\frac{k}{2}$ corresponding to the diagonal.

%WE MAY PUT HERE ONE OF THE FIGURES (EMBEDDED OR FHJGraph) WHERE THE DIAGONAL STEP IS ALSO SHOWN.
\noindent 
one must borrow from the reactions $\cf{2X} \to \cf{0}$ and $\cf{2X} \to \cf{2X+2Y}$ in the sense that
 \begin{align*}
 \underbrace{
     k_1 x^2 \begin{pmatrix}
        -2 \\ 0
     \end{pmatrix}
% }_{\mathrm{2X}\to \mathrm{0}}
    %%
    \,\,+\,\,
%  \underbrace{
     k_2 x^2 \begin{pmatrix}
        0 \\ 2
     \end{pmatrix}
% }_{\mathrm{2X}\to \mathrm{2X+2Y}}
}_{\text{in first network}}
    =
 \underbrace{
     (k_1-\epsilon) x^2 \begin{pmatrix}
        -2 \\ 0
     \end{pmatrix}
% }_{\mathrm{2X}\to \mathrm{0}}
    %%
    \,\,+\,\,
%  \underbrace{
    ( k_2-\epsilon) x^2 \begin{pmatrix}
        0 \\ 2
     \end{pmatrix}
% }_{\mathrm{2X}\to \mathrm{2X+2Y}}
%%
    \,\,+\,\,
%  \underbrace{
     2\epsilon x^2 \begin{pmatrix}
        -1 \\ 1
     \end{pmatrix}
% }_{\mathrm{2X}\to \mathrm{X+Y}}
}_{\text{in second network}}
 \end{align*}
where $0 < \epsilon < \min\{k_1, k_2\}$.
\end{example}

The example above may lead us to believe that it is always possible to add a reaction from some $\yb_n$ on the boundary of the convex hall to some complex $\yb$ in the interior of the convex hull of other complexes. The example below illustrates that this is not always the case, and further work is needed to explore when it is possible to add such a reaction.

\begin{example}
If one takes the system
 
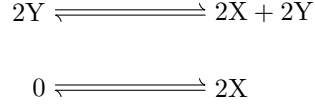
\begin{center}
    \begin{tikzpicture}
        \node (0) at (0,0) [left] {0};
        \node (2x) at (2,0)  [right] {\cf{2X}};
        \node (2x2y) at (2,1) [right] {\cf{2X + 2Y}};
        \node (2y) at (0,1) [left] {\cf{2Y}};
        \draw [-{>[harpoon]}, transform canvas={yshift=1}] (0)--(2x);
        \draw [-{>[harpoon]}, transform canvas={yshift=-1}] (2x)--(0);
        \draw [-{>[harpoon]}, transform canvas={yshift=1}] (2y)--(2x2y);
        \draw [-{>[harpoon]}, transform canvas={yshift=-1}] (2x2y)--(2y);
    \end{tikzpicture}
    \end{center}
and tries to add the reaction step $\cf{2X} \to \cf{X+Y}$ (note that the new complex $\cf{X+Y}$ is a convex combination of the old ones), there is no way of choosing the rate coefficients so as to get the induced kinetic differential equation of the extended one. 
\end{example}

%%%%%%%%%%%%%%%%%%%%%%%%%%%%%%%%%%%%%%%%%%%%%%%%%%%%%%%%%%%%%%%%%%%%%%%%%
\bigskip 

\subsubsection{Eliminatin complexes from a realization}

Perhaps instead of adding or removing reactions from an existing realization, one may want to add or remove complexes from an existing realization. As noted in~\cite{szederkenyiComment}, it is always possible to add a complex that has not appeared in the current realization, as long as the weighted sum of reaction vectors is $\vv 0$. For example, one can always add the reactions $\cf{0} \xleftarrow{k} \cf{X} \xrightarrow{k} \cf{2X}$ without changing the induced kinetic differential equation. The question lies, then, on when it is possible to eliminate a complex from your realization. 

\begin{theorem}
Suppose that $(\Yb, \Ab_{k})$ is a realization with $N+1$ complexes $\yb_1,\dots,\yb_N, \yb_{N+1}$,
where the complex $\yb_{N+1}$ participates in at least one reaction as a reactant complex.
Suppose that the monomial $\xb^{\yb_{N+1}}$ does not appear explicitly in the kinetic differential equation.
Then the differential equation can also be realized using only the complexes
$\yb_1,\dots,\yb_N$.
\end{theorem}
\begin{proof}
We can write the complex matrix and Kirchoff matrix  as
\begin{equation*}\label{eq:red_matrices}
\Yb = \left[
  \begin{array}{cc}
    \Yb' & \yb_{N+1}
\end{array} \right], 
\quad \text{and}\quad 
  \Ab_{k} = \left[
   \begin{array}{cc} 
     \Ab'_{k} & \vb \\
%     \ub\T & -\sum_{n=1}^{N}\vb_n
   \end{array} 
  \right],
\end{equation*}
where $\ub, \vb \in \mathbb{R}^{N}_{\geq 0}$.
Then the kinetic differential equation can be written as
\begin{equation*}
  \frac{d\xb}{dt} = \Yb'\Ab'_{k} \xb^{\Yb'}+\Yb'\vb \xb^{\yb_{N+1}}+\yb_{N+1}\ub\T \xb^{\Yb'}-\left(\sum_{n=1}^N\vb_n\right)\yb_{N+1}\xb^{\yb_{N+1}}.
\end{equation*}
Since the complex $\yb_{N+1}$ admits some outgoing reactions, we have that $\sum_{n=1}^N\vb_n\neq 0$, and,
using the fact that the coefficient of $\xb^{\yb_{N+1}}$ is zero, we find $\yb_{N+1}=\frac{1}{\sum_{n=1}^N\vb_n} \Yb'\vb$, and
\begin{equation*}
  \frac{d\xb}{dt} = \Yb'\left(\Ab'_{k}+\frac{1}{\sum_{n=1}^N\vb_n}\vb\ub\T \right) \xb^{\Yb'}.
\end{equation*}
We conclude that $\yb_{N+1}$ is a convex combination of the remaining complexes. More importantly, the Kirchoff matrix $\Ab''_{k}=\Ab'_{k}+\frac{1}{\sum_{n=1}^N\vb_n}\vb\ub\T $ gives a realization of the same kinetic equation with only the complexes $\yb_1,\dots,\yb_N$.
\end{proof}
\begin{remark}
If $\xb^{\yb}$ appears explicitly as a term in the kinetic differential equation, then that complex cannot be removed. 
\end{remark}

%%%%%%%%%%%%%%%%%%%%%%%%%%%%%%%%%%%%%%%%%%%%%%%%%%%%%%%%%%%%%%%%%%%%%%%%%%%%%%%%%%%%%
\bigskip 

After seeing the procedure of eliminating extraneous complexes, one may be tempted to remove all complexes that do not appear as exponents of monomials without ruining desirable properties (e.g., weak reversibility, complex balanced). 

Indeed, if one is interested in finding a  reversible, or weakly reversible, or detailed balanced, or complex balanced realization, any complex that does not appear in the differential equation (after simplification) can be eliminated from the network~\cite{craciunjinyu}: 
\begin{theorem}\label{thm:CBnonewcomplex}
    A kinetic differential equation $\dot{\xb} = \mathbf{f}(\xb)$ has a complex balanced realization if and only if it has a complex balanced realization where the complexes (the columns of $\Yb$) are precisely the monomials in $\mathbf{f}(\xb)$. Analogous results hold for detailed balanced, weakly reversible, and reversible realizations. 
\end{theorem}

As a consequence, the algorithms described in Section~\ref{sec:algorithms} can be used to determine whether a given system of kinetic differential equations admits a complex balanced (or detailed balanced, weakly reversible, or reversible) realization. This is a finite calculation as there is no need to compute using different sets of complexes.

%%%%%%%%%%%%%%%%%%%%%%%%%%%%%%%%%%%%%%%%%%%%%%%%%%%%%%%%%%%%%%%%%%%%%%%%%%%%%%%%%%%%%%%%%%%%%%%%%%%%%%%%%%%%%%%%%%%%%%%%%%%%%%%%%%%%%%%%%%%%%%
\section{Relation between \WR\ and \CB\ realizations for symbolic kinetic equations}
\label{sec:WRandCB}
%%%%%%%%%%%%%%%%%%%%%%%%%%%%%%%%%%%%%%%%%%%%%%%%%%%%%%%%%%%%%%%%%%%%%%%%%

When we discussed the algorithm for computing weakly reversible realizations, the reader may have noticed that we introduced an additional Kirchoff matrix $\Ab_k'$ that does not appear in the realization $(\Yb, \Ab_k)$.  This reflects the fact that there is a subtle algebraic relation between weak reversibility and complex balancing. In this section, we explore this relationship.

In order to study the possible network structures that can arise from kinetic differential equation to those with unspecified coefficient, we make the following definition:
\newcommand{\kincomptb}{kinetically compatible } %%%%%%%%%%%
\begin{definition}
  For a matrix $\Yb \in \zz_{\geq 0}^{M\times N}$ we say that a sign matrix $\Zb^\sigma \in \{0,+,-\}^{M \times N}$ is \df{\kincomptb with $\Yb$} if $\Yb_{mr}=0$ implies  $\Zb^\sigma_{mr} \in\{
0, +\}$.  In such a case, let 
  \begin{align}
      \mc Z(\Yb, \Zb^\sigma) := \left\{ \Zb\in\mathbb{R}^{M\times N} \colon \mathrm{sign}(\Zb) = \Zb^\sigma \right\},
\label{eq:kinetic_RHS}
  \end{align}
 where $\sign(\cdot)$ is the componentwise sign function. 
\end{definition}

For any $\Yb\in\zz_{\geq 0}^{M\times N}$ and \kincomptb $\Zb^\sigma$, the set $\mathcal{Z}(\Yb, \Zb^\sigma)$ is an unbounded convex cone in $\R^{M\times N}$.  
According to~\cite{harstoth}, the differential equation 
$
\dot{\xb}=\Zb\cdot \xb^\Yb$
 is kinetic if $\Zb \in \mc Z(\Yb, \Zb^\sigma)$ for any \kincomptb  $\Zb^\sigma$.

\begin{example}
The Lotka-Volterra equations 
    \begin{align*}
        \dot{x} &= \alpha x - \beta xy \\
        \dot{y} &= \delta xy - \gamma y
    \end{align*}
(with $\alpha, \beta, \gamma, \delta > 0$) can be written in this framework. Let 
\begin{align*}
    \Yb = \begin{pmatrix}
        1 & 1 & 0 \\
        0 & 1 & 1
    \end{pmatrix}
    \quad \text{and} \quad 
    \Zb^\sigma = \begin{pmatrix}
        + & - & 0 \\
        0 & + & -
    \end{pmatrix},
\end{align*}
where clearly $\Zb^\sigma$ is kinetically compatible with $\Yb$. The matrix corresponding to the Lotka-Volterra equations is
    $
        \Zb = \begin{pmatrix}
            \alpha & -\beta & 0 \\
            0 & \delta & -\gamma
        \end{pmatrix}
    $, 
which belongs to the set $\mc Z(\Yb, \Zb^\sigma)$.
\end{example}

% \begin{example}
% The Lorenz equation 
% \begin{align*}
%     \dot{x} &= \sigma y - \sigma x \\
%     \dot{y} &= \rho x - xz - xy \\
%     \dot{z} &= xy - \beta z
% \end{align*}
%  is not kinetic for any choice of $\sigma, \rho, \beta > 0$ because of the term $xz$ in $\dot{y}$. To phrase it in terms of the above definition, let $\Yb=\begin{pmatrix}
% 1&0&0&1&1\\
% 0&1&0&0&1\\
% 0&0&1&1&0
% \end{pmatrix}$, and $\Zb^\sigma = 
% \begin{pmatrix}
%     -&+&0&0&0 \\
%     +&-&0&-&0 \\
%     0&0&-&0&+
% \end{pmatrix}$. Then for any $\sigma, \rho, \beta  > 0$, the resulting kinetic matrix $\Zb$ satisfies the sign condition $\mathrm{sign}(\Zb) = \Zb^\sigma$ even though $\Zb^\sigma$ is not \kincomptb with $\Yb$.
% % $\Zb(\Yb)=\left(\begin{array}{ccccc}
% % -\sigma&\sigma&0&0&0\\
% % \varrho&-1&0&-1&0\\
% % 0&0&-\beta&0&1
% % \end{array}\right)$ (with $\beta,\varrho,\sigma>0$) the component $z(\Yb)_{24}=-1$ violates the condition \eqref{eq:kinetic_RHS}, because $y_{24}=0$, thus $\Zb(\Yb)\notin\mathcal{Z}$.
% \end{example}

We defined $\mc Z(\Yb, \Zb^\sigma)$ in order to study possible behaviours a kinetic differential equation may exhibit for some choice of rate coefficients.

\begin{definition}
    Let $\Zb^\sigma$ be \kincomptb with the matrix $\Yb$. 
  We say that $\mc Z(\Yb, \Zb^\sigma)$ has the \df{capacity to be complex balanced} (or \df{capacity to be weakly reversible}) if there exists $\Zb \in \mc Z(\Yb, \Zb^\sigma)$ such that $\dot{\xb} = \Zb \cdot \xb^{\Yb}$ has a complex balanced (or weakly reversible) realization. 
\end{definition}

\begin{theorem}
    Let $\Zb^\sigma$ be \kincomptb with the matrix $\Yb$. The set $\mc Z(\Yb, \Zb^\sigma)$ has the capacity to be complex balanced if and only if it has the capacity to be weakly reversible. 
\end{theorem}
\begin{proof}
    It is known that complex balancing implies weak reversibility. Suppose there is some $\Zb \in \mc Z(\Yb, \Zb^\sigma)$ with a weakly reversible realization $(\Yb, \Ab_k)$. Then there is a positive vector $\vv p$ in the kernel of the Kirchoff matrix $\Ab_k$. Fix an arbitrary positive state $\xb^*$. Consider the following scaled Kirchoff matrix
    \begin{align}\label{eq:WRCBscale}
        \Ab_k' &= \Ab_k \cdot \diag\left( \frac{\vv p}{(\xb^*)^{\Yb}} \right)
    \end{align}
    where the division $\frac{\vv p}{(\xb^*)^{\Yb}}$ is componentwise. Then 
    \begin{align*}
        \Ab_k' (\xb^*)^{\Yb} = \Ab_k \cdot  \diag\left( \frac{\vv p}{(\xb^*)^{\Yb}} \right) \cdot (\xb^*)^{\Yb}
        = \Ab_k\cdot \vv p = \vv 0, 
    \end{align*}
i.e., $\Ab_k'$ defines a complex balanced realization with stationary concentration $\xb^*$.
\end{proof}

\section{On deficiency zero realizations and uniqueness}
\label{sec:uniqueness}

Recall that the deficiency of a reaction network is the nonnegative integer $\delta = N - L - S$, where $N$ is the number of complexes, $L$ is the number of linkage classes, and $S$ is the dimension of the stoichiometric subspace $\SSS$. A network with lower deficiency tends to have nicer dynamical properties; for example, by the Deficiency Zero Theorem or Deficiency One Theorem. Let us formulate a necessary and sufficient condition on network structure for zero deficiency~\cite{feinbergbook}.

We say that the vectors $\yb_0,\yb_1,\dots,\yb_N\in\R^M$ are \df{affinely independent} if $\yb_1-\yb_0,\dots,\yb_N-\yb_0$ are linearly independent. Note that in this case $\yb_0-\yb_n,\dots, \yb_i - \yb_n, \dots, \yb_N-\yb_n$ are linearly independent for any $n=1,2,\dots,N$ and $i \neq n$. Also, linearly independent vectors are affinely independent.

\begin{theorem}\label{thm:zerodef}
A reaction network is deficiency zero if and only if
\begin{enumerate}
\item
the complex vectors within each linkage class are affinely independent, and
\item
the stoichiometric spaces belonging to the linkage class are linearly independent.
\end{enumerate}
\end{theorem}
\begin{proof}
%Let $\mathcal{L}_1, \mathcal{L}_2, \dots, \mathcal{L}_L$ be the linkage classes of the reaction network; 
Consider the $\ell$-th linkage class. Let $N_\ell$ be the number of complexes in this  linkage class, and let $\SSS_\ell$ be the stoichiometric space belonging to this linkage class. Let $\delta_\ell := N_\ell - 1 - \dim \SSS_\ell$ be the deficiency of the $\ell$-th linkage class. Finally, let $\SSS$ be the full stoichiometric space.

The deficiency of the entire network can be related to the deficiencies of the linkage classes:
\begin{align*}
\delta &= N - L - S
       = \sum_{\ell=1}^L (N_\ell - 1 - \dim \SSS_\ell) + \sum_{\ell=1}^L \SSS_\ell - \dim \SSS\\
       &= \sum_{\ell=1}^L \delta_\ell + \sum_{\ell=1}^L \dim \SSS_\ell - \dim \SSS
       \geq \sum_{\ell=1}^L \delta_\ell.
\end{align*}
Equality holds if and only if the stoichiometric spaces belonging to the linkage classes $\{\SSS_\ell\}_{\ell=1}^L$ are linearly independent. Moreover, $\delta = 0$ if and only if all $\delta_\ell =0$. Finally, $\delta_{\ell}=0$ implies $\dim \SSS_\ell=N_\ell-1$, i.e. the complexes in any given linkage classes are affinely independent.
\end{proof}

If we restrict ourselves to weakly reversible realizations, the deficiency $\delta$ is minimized if the set of complexes are precisely the exponents of monomials~\cite{craciunjinyu}. Each additional complex that does not appear as exponents of monomials increases the deficiency $\delta$. It has been shown that any deficiency zero realization with one strong terminal class is unique if the set of complexes is fixed~\cite{Csercsik2012a}.  Combining these two facts, we have the following Proposition.
\begin{proposition}
\label{corollary:def0_1lc}
Any weakly reversible deficiency zero realization with a single linkage class is unique.
\end{proposition}

The hypotheses of weakly reversible and deficiency zero in Proposition~\ref{corollary:def0_1lc} are required as the following examples illustrate.
\begin{example}
    The reaction networks $\cf{Y} \xleftarrow{k} \cf{0} \xrightarrow{k} \cf{X}$ and $\cf{0} \xrightarrow{k} \cf{X+Y}$ are two different deficiency zero realizations with one linkage class, although they are not weakly reversible.
    %  \begin{center}
    % \begin{tikzpicture}
    % \begin{scope}[transform canvas={xshift=0cm}]
    %     \node (0) at (0,0) {\cf{0}};
    %     \node (x) at (2,0) {\cf{X}};
    %     \node (y) at (0,2) {\cf{Y}};
    %     \draw [->] (0)--(x) node [midway, below] {$k$};
    %     \draw [->] (0)--(y) node [midway, left] {$k$};
    % \end{scope}
    % %%%%%%%%%%%%%%%
    % \begin{scope}[transform canvas={xshift=4cm}]
    %     \node (0) at (0,0) {\cf{0}};
    %     % \node (x) at (2,0) {\cf{X}};
    %     % \node (y) at (0,2) {\cf{Y}};
    %     \node (xy) at (2,2) {\cf{X+Y}};
    %     \draw [->] (0)--(xy) node [midway, above left] {$k$};
    % \end{scope}
    % \end{tikzpicture}
    % \end{center}
\end{example}

\begin{example}
    The reaction network $\cf{0} \rightleftharpoons \cf{X} \rightleftharpoons \cf{2X}$ (with any rate coefficients) and the reaction network 
    
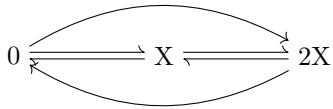
\begin{center}
    \begin{tikzpicture}
        \node (0) at (0,0) {\cf{0}};
        \node (x) at (2,0) {\cf{X}};
        \node (2x) at (4,0) {\cf{2X}};
        \draw [-{>[harpoon]}, transform canvas={yshift=1.25pt}] (0)--(x);
        \draw [-{>[harpoon]}, transform canvas={yshift=-1.25pt}] (x)--(0);
        \draw [-{>[harpoon]}, transform canvas={yshift=1.25pt}] (x)--(2x);
        \draw [-{>[harpoon]}, transform canvas={yshift=-1.25pt}] (2x)--(x);
        \draw [->] (0)  to [bend left] (2x);
        \draw [->] (2x)  to [bend left] (0);
    \end{tikzpicture}   
    \end{center}
(for wisely chosen rate coefficients) are dynamically equivalent. Both are weakly reversible and with one linkage class, but $\delta = 1$. 
\end{example}

\begin{example}
    Two dynamically equivalent weakly reversible realizations where one has $\delta = 1$ are
    
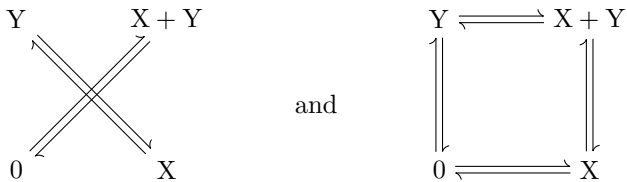
\begin{center}
    \begin{tikzpicture}
        \node (0) at (0,0) {\cf{0}};
        \node (x) at (2,0) {\cf{X}};
        \node (y) at (0,2) {\cf{Y}};
        \node (xy) at (2,2) {\cf{X+Y}};
        \draw [-{>[harpoon]}, transform canvas={yshift=1pt, xshift=-1pt}] (0)--(xy);
        \draw [-{>[harpoon]}, transform canvas={yshift=-1pt, xshift=1pt}] (xy)--(0);
        \draw [-{>[harpoon]}, transform canvas={yshift=-1pt, xshift=-1pt}]  (x)--(y);
        \draw [-{>[harpoon]}, transform canvas={yshift=1pt, xshift=1pt}]  (y)--(x);
    \end{tikzpicture}
    \hspace{0.75cm}
     \begin{tikzpicture}
        \node at (0,0) {};
        \node at (0,1) {and};
     \end{tikzpicture}
    \hspace{0.75cm}
    %%%%%%%%%%%%%%%
    \begin{tikzpicture}
        \node (0) at (0,0) {\cf{0}};
        \node (x) at (2,0) {\cf{X}};
        \node (y) at (0,2) {\cf{Y}};
        \node (xy) at (2,2) {\cf{X+Y}};
        \draw [-{>[harpoon]}, transform canvas={yshift=1.25pt}] (0)--(x);
        \draw [-{>[harpoon]}, transform canvas={yshift=-1.25pt}] (x)--(0);
        \draw [-{>[harpoon]}, transform canvas={xshift=-1.25pt}] (x)--(xy);
        \draw [-{>[harpoon]}, transform canvas={xshift=1.25pt}] (xy)--(x);
        \draw [-{>[harpoon]}, transform canvas={yshift=-1.25pt}] (xy)--(y);
        \draw [-{>[harpoon]}, transform canvas={yshift=1.25pt}] (y)--(xy);
        \draw [-{>[harpoon]}, transform canvas={xshift=-1.25pt}] (0)--(y);
        \draw [-{>[harpoon]}, transform canvas={xshift=1.25pt}] (y)--(0);
    \end{tikzpicture}
    \end{center}
with all rate coefficients are taken to be $1$.
\end{example}

Nonetheless, Proposition~\ref{corollary:def0_1lc} naturally leads us to the following conjecture, for which a proof has been proposed~\cite{CraciunJinYu_uniqueness}:
\begin{conjecture}
    Any weakly reversible deficiency zero reaction network has a unique realization.
\end{conjecture}

\bigskip \bigskip

Throughout this work, we have seen over and over again that realizability (more precisely, identifying the network structure) is an underdetermined problem. Is it possible to use this flexibility in network structure to get more out information about a kinetic differential equation?

An idea comes to mind; we will formulate a conjecture following an example.
\begin{example}\label{ex:conj}
Consider the mass action system
    
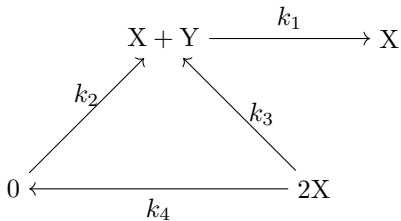
\begin{center}
    \begin{tikzpicture}
        \node (0) at (0,0) {\cf{0}};
        \node (2x) at (4,0) {\cf{2X}};
        \node (xy) at (2,2) {\cf{X+Y}};
        \node (x) at (5,2) {\cf{X}};
        \draw [->] (0)--(xy) node [midway, above] {$k_2$};
        \draw [->] (2x)--(0) node [midway, below] {$k_4$};
        \draw [->] (2x)--(xy) node [midway, right] {$k_3$};
        \draw [->] (xy)--(x) node [midway, above] {$k_1$};
    \end{tikzpicture}   
    \end{center}
\noindent 
for any rate coefficients $k_1$, $k_2$, $k_3$, $k_4 > 0$. This system has a weakly reversible, deficiency zero realization:
    
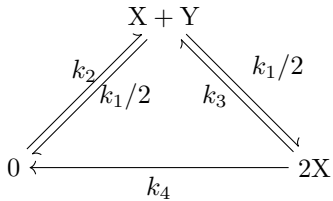
\begin{center}
    \begin{tikzpicture}
        \node (0) at (0,0) {\cf{0}};
        \node (2x) at (4,0) {\cf{2X}};
        \node (xy) at (2,2) {\cf{X+Y}};
        \draw [-{>[harpoon]}, transform canvas={xshift=-1pt, yshift=1pt}] (0)--(xy) node [midway, above] {$k_2$};
        \draw [-{>[harpoon]}, transform canvas={xshift=1pt, yshift=-1pt}] (xy)--(0) node [midway, right] {$k_1/2$};
        \draw [->] (2x)--(0) node [midway, below] {$k_4$};
        \draw [-{>[harpoon]}, transform canvas={xshift=-1pt, yshift=-1pt}] (2x)--(xy) node [midway, left] {$k_3$};
        \draw [-{>[harpoon]}, transform canvas={xshift=1pt, yshift=1pt}] (xy)--(2x) node [midway, above right] {$k_1/2$};
    \end{tikzpicture}   
    \end{center}

The induced kinetic differential equation of both systems is
$$
\dot{x}=k_2-x^2(k_3+2k_4), \qquad \dot{y}=k_2-k_1xy+k_3x^2.
$$
In particular, the latter mass action system is complex balanced for all choice of rate coefficients. This implies that the kinetic differential equation enjoys all the dynamical properties of complex balanced systems, although we could have been looking at the network that is not weakly reversible.
\end{example}

\begin{conjecture}\label{conj:codimvar}
Consider a reaction network $\NN_1$, and assume that there exists a \WR\ reaction network $\NN_2$ of deficiency $\delta$ such that any \de\ induced by $\NN_1$ and \emph{positive} $\kb$ values can also be induced by $\NN_2$ and positive $\kbh$ values, and viceversa.
Then, the set of rate coefficient values of $\NN_1$ for which $\NN_1$ induces systems that have \CB\ realizations contains an algebraic variety of codimension not larger than $\delta$.
\end{conjecture}

%%%%%%%%%%%%%%%%%%%%%%%%%%%%%%%%%%%%%%%%%
\section{Discussion}
%%%%%%%%%%%%%%%%%%%%%%%%%%%%%%%%%%%%%%%%%

We have been talking about finding realizations (with certain network properties) of kinetic differential equations. Starting with a system of polynomial differential equations (possibly with symbolic parameters), one could try to find a (Feinberg-Horn-Jackson) reaction graph  that induces the differential equation under mass action kinetics. After fixing a set of complexes, algorithms exist to find realizations (Section~\ref{sec:algorithms}). Furthermore, conditions on the resulting reaction graph (e.g., weakly reversible, complex balanced, omit certain reactions) can be incorporated into the algorithms.\footnote{Codes for solving the problems in the paper can be requested from the authors.}

We also considered graph operations on the reaction network that will preserve the kinetic differential equation. For example, in certain cases reactions can be added, and complexes can be removed. Our discussion in Section~\ref{sec:modnetworkDE} is far from complete; it would be interesting to further explore the allowed operations while preserving dynamical equivalence. 

In Section~\ref{sec:WRandCB}, we explored the relationship between weakly reversible and complex balanced realizations. Hidden inside the proof of the main theorem in that section is the scaling equation \eqref{eq:WRCBscale}. We believe there is information hidden about the manifold in parameter space that decides whether or not a kinetic differential equation can be complex balanced. 

Finally, in our discussion of deficiency zero realizations and uniqueness, we posed two conjectures. The first is that weakly reversible deficiency zero networks have unique realization. The latter is more complicated; it asks for what rate coefficients is the reaction network \emph{dynamically equivalent to complex balanced}. 

As the reader can see, a theoretical understanding of the relationships between network structure, the sign structure of the differential equation, and dynamical property is lacking. We hope to see further studies exploring the realizability problem.

%%%%%%%%%%%%%%%%%%%%%%%%%%%%%%%%%%%%%%%%%%%%%%%%%%%%%%%%%%%%%%%%%%%%
\subsection*{Acknowledgement}
%%%%%%%%%%%%%%%%%%%%%%%%%%%%%%%%%%%%%%%%%%%%%%%%%%%%%%%%%%%%%%%%%%%%
The authors started to work on this paper when enjoying the hospitality of the Erwin Schr\"odinger Institut in Vienna as participants of the meeting Advances in Chemical Reaction Network Theory, 15--19 October, 2018. G\'abor Szederk\'enyi and J\'anos T\'oth thank the support of the National Research, Development, and Innovation Office, Hungary (SNN 125739). Gheorghe Craciun and Polly Y. Yu have been supported by National Science Foundation grant DMS-1816238. Polly Y. Yu also by the Natural Sciences and Engineering Research Council of Canada. Elisa Tonello was supported by the Volkswagen Stiftung (Volkswagen Foundation) under the funding initiative Life? - A fresh scientific approach to the basic principles of life (project ID: 93063). We utilized the detailed comments of the reviewers gratefully.

\bibliography{Inverse}
\bibliographystyle{plain}%{apalike}%{spmpsci}%{plain}%{spphys}%
\nocite{*}

%%%%%%%%%%%%%%%%%%%%%%%%%%%%%%%%%%%%%%%%%%%%%%%
\section*{Supplementary}
%%%%%%%%%%%%%%%%%%%%%%%%%%%%%%%%%%%%%%%%%%%%

%%%%%%%%%%%%%%%%%%%%%%%%%%%%%%%%%%%%%%%%%%%%%%%
\subsection*{Examples: Realizations with a minimal number of complexes}
%%%%%%%%%%%%%%%%%%%%%%%%%%%%%%%%%%%%%%%%%%%%

%%%%%%%%%%%%%%%%%%%%%%%%%%%%%%%%%%%%%%%%%%%%%%%
Simple examples can be treated directly, without detailed application of Theorem \ref{thm:general}.
\begin{example}
If the \de\ $\dot{x}=kx$ is given (no matter what the sign of $k$ is), then we might try to look for an inducing reaction network consisting of the single complex \ce{X} and find it impossible the minimal reason being that one cannot have a reaction network with less than two complexes. The right hand side contains a single monomial, but the realization should contain more than one complexes.
\end{example}
\begin{example}\label{ex:fourcomplexes}
If the \de\
\begin{equation}\label{eq:fourcomplexes}
\dot{x}=-k_1x+2k_2y,\quad \dot{y}=3k_1x-k_2y
\end{equation}
is given with $k_1,k_2>0$, then at least four complexes are needed to realize the equation. How do we prove this?
\begin{description}
\item[Two complexes are not enough.]
With the notations of Theorem \ref{thm:general} we have
$$
\Zb=\left(\begin{array}{cc}
-k_{1} & 2k_{2} \\
3k_{1} & -k_{2} \\
\end{array}\right),
\Yb=\left(\begin{array}{cc}
1 & 0 \\
0 & 1 \\
\end{array}\right),
\bb=\nulb\in\R^2.
$$
As $\Yb$ is invertible, $\vb=\Yb^{-1}\nulb=\nulb$ should hold. However,
$\Yb\Bb=\Bb=\Zb$ cannot hold, because $\Zb$ is not a compartmental matrix.

\item[Three complexes are not enough either.]
Include $(\begin{array}{cc}\alpha_1 & \alpha_2\end{array})\T$ among the complex vectors with unknown nonnegative integers $\alpha_1$ and $\alpha_2$ so that this vector is different from the other two we have as yet included. Then our task is to find nonnegative rate coefficients $k_{21},k_{31},k_{12},k_{32},k_{13},k_{23}$ (at least two of them positive) so that
%{\footnotesize
{\small
\begin{align*}
&\left(\begin{array}{ccc}
1 & 0 &\alpha_1\\
0 & 1 &\alpha_2\\
  \end{array}\right)
\left(\begin{array}{ccc}
-k_{12}-k_{13} & k_{21}        & k_{31} \\
k_{12}         & -k_{21}-k_{23}& k_{32} \\
k_{13}         & k_{23}        & -k_{31}-k_{32}
\end{array}\right)
\left(\begin{array}{c}
  x \\
  y \\
  x^{\alpha_1}y^{\alpha_2}
\end{array}\right)\\
&=
\left(\begin{array}{c}
 -(k_{12} + k_{13} - k_{13} \alpha_1)x + (k_{21} + k_{23} \alpha_1)y +
   (k_{31} -(k_{31} + k_{32}) \alpha_1)x^{\alpha_1} y^{\alpha_2}  \\
   (k_{12} + k_{13} \alpha_2)x
   -(k_{21} + k_{23} - k_{23} \alpha_2)y +  (k_{32} -(k_{31} + k_{32}) \alpha_2)x^{\alpha_1} y^{\alpha_2}
\end{array}\right).
\end{align*}
}
First of all, one should have zero coefficients before $x^{\alpha_1}y^{\alpha_2}$,
thus
$
  \alpha_1=\frac{k_{31}}{k_{31}+k_{32}}\quad\alpha_2=\frac{k_{32}}{k_{31}+k_{32}},
$
should hold, therefore
$
  1-\alpha_1=\alpha_2\quad 1-\alpha_2=\alpha_1.
$
Equating the coefficients of $x$ and $y$ here with those of the right hand side gives an equation without solution for $k_{12},\ldots,\alpha_2$ so that $\alpha_1$ and $\alpha_2$ are nonnegative integers, and among the nonnegative numbers $k_{nq}$ at least two are positive.
The equations are as follows.
\begin{align}
&k_{12}+k_{13}(1-\alpha_1)&=k_{12}+k_{13}\alpha_2=&k_1, \label{eq:3complex1}\\
&k_{21}+k_{23}\alpha_1    &=                      &2k_2,\label{eq:3complex2}\\
&k_{12}+k_{13}\alpha_2    &=                      &3k_1,\label{eq:3complex3}\\
&k_{21}+k_{23}(1-\alpha_2)&=k_{21}+k_{23}\alpha_1=&k_2. \label{eq:3complex4}
\end{align}
Obviously,
\eqref{eq:3complex1} and \eqref{eq:3complex3} could only hold if $k_1=0$, which cannot happen. Similarly,
\eqref{eq:3complex2} and \eqref{eq:3complex4} could only hold if $k_2=0$, which cannot happen either.
\item[Four complexes are enough.]
Finally, one can easily calculate that the complex vectors in the complex matrix
$$
\left(\begin{array}{cccc}
 1 & 0 & 0 & 2 \\
 0 & 1 & 3 & 0
\end{array}\right)
$$
are enough, Eq. \eqref{eq:fourcomplexes} is the \ikde\ of the reaction network
$$\ce{X ->[$k_1$] 3Y, \quad Y ->[$k_2$] 2X}.$$
Note that this realization has the minimal number of reactions,
(one reaction step is not enough),
and is of deficiency zero, as well.
\end{description}
Summing up: here the right hand side of \eqref{eq:fourcomplexes} contains two monomials, still there exists no realization with two complexes, more are needed.

Furthermore,  three complexes are enough if we allow complexes in $\rr_{\geq 0}^M$ instead of $\zz_{\geq 0}^M$. In particular, $\alpha_1 = \frac 45, \alpha_2 = \frac 35$ is an appropriate choice.  In the pioneering paper~\cite{hornjackson} this generalization to non-negative real coefficients for complexes is allowed, although here we restrict ourselves to non-negative integer coefficients.
\end{example}

\subsection{More examples}

\begin{example}\label{ex:yinverse}
Invertibility of the complex matrix $\Yb$ is not sufficient to have a realization with the exponents of the monomials as complex vectors, as the following example shows.
$$
\dot{x}=x^3-xy^2, \dot{y}=2x^3
$$
implies that
$$
\Yb=\left(\begin{array}{cc}
3 & 1 \\
0 & 2 \\
\end{array}\right),\quad
\Zb=\left(\begin{array}{cc}
1 & -1 \\
2 & 0 \\
\end{array}\right)
$$
and the (unique) solution to
$$
\left(\begin{array}{cc}
3 & 1 \\
0 & 2 \\
\end{array}\right)
\left(\begin{array}{cc}
-b_{21}-u_1 & b_{12} \\
b_{21} & -b_{12}-u_2 \\
\end{array}\right)
=\left(\begin{array}{cc}
1 & -1 \\
2 & 0 \\
\end{array}\right)
$$
is
$$
b_{12}=-1/3, b_{21}=1, u_1=-1, u_2=1/3,
$$
with some negative numbers.
We are forced to formulate a problem: Characterize those matrices with nonnegative (integer) components the inverse of which multiplied with a matrix without negative cross effect gives a compartmental matrix. As this problem again requires finding the feasible solutions of a linear programming (more precisely, MIL) problem, there is not too much hope to find an explicit symbolic solution, but at the same time one can have numerical solution(s) using LP-solvers.
\end{example}
\begin{example}\label{ex:canonical}
Consider the polynomial equation
\begin{equation}\label{eq:poly}
\dot{x}=k_1x^2y^5z-k_{-1}x^3y^4z\quad
\dot{y}=-k_1x^2y^5z+k_{-1}x^3y^4z\quad
\dot{z}=0
\end{equation}
with $k_1,k_{-1}>0$.
The canonical representation of this equation is
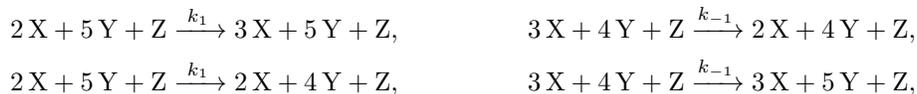
\begin{align*}
&\ce{2X + 5Y + Z ->[$k_1$] 3X + 5Y + Z},&\ce{3X + 4Y + Z ->[$k_{-1}$] 2X + 4Y + Z},\\
&\ce{2X + 5Y + Z ->[$k_1$] 2X + 4Y + Z},&\ce{3X + 4Y + Z ->[$k_{-1}$] 3X + 5Y + Z},
\end{align*}
a weakly reversible reaction network with $N=4$ complexes, $L=1$ linkage class and a $S=2$-dimensional stoichiometric space, thus it has a deficiency $1=N-L-S$.

Direct calculation shows that the only positive stationary point is
$$
x_*=\frac{x_0+y_0}{1+K}\quad y_*=\frac{K(x_0+y_0)}{1+K}\quad z_*=z_0,
$$
however, its existence directly follows from weak reversibility~\cite{borosPosSteadyStates}.

\end{example}
\begin{example}\label{ex:opencc}
The canonical realization of the \de\
\begin{equation}\label{eq:opencc}
\dot{x}=-4x+2y+3z\quad\dot{y}=x-5y+z+1\quad\dot{z}=2x+3y-4z
\end{equation}
is the irreversible, deficiency three reaction network consisting of a single linkage class of Figure~\ref{fig:canonicalfirst}.
\begin{figure}[!h]
  \centering
  \begin{tikzcd}[ampersand replacement=\&,row sep=large,column sep=small]
    \& \cf{X+Y} \& \& \cf{Y} \arrow[ll,"2"'] \arrow[dr,"3"] \arrow[ld,rightharpoonup,yshift=-1,"5"] \& \\
    \cf{X} \arrow[ru,"1"] \arrow[rr,"4"] \arrow[dr,"2"'] \& \& \cf{0} \arrow[ru,rightharpoonup,yshift=4pt,"1"] \& \& \cf{Y+Z} \\
    \& \cf{X+Z} \& \& \cf{Z} \arrow[ll,"3"] \arrow[lu,"4"] \arrow[ru,"1"'] \&
  \end{tikzcd}
  \caption{The canonical realization of Eq. \eqref{eq:opencc}}\label{fig:canonicalfirst}
\end{figure}
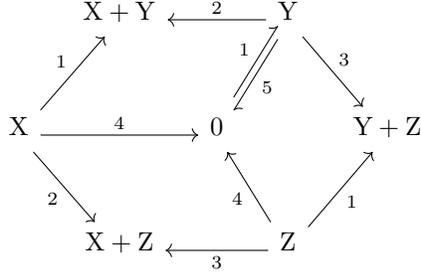
However, it is "obviously"
the \ikde\ of the zero deficiency weakly reversible reaction network in Figure~\ref{fig:opencc1}. (Here again, the unique positive stationary point can explicitely be determined.)
\begin{figure}[!h]
  \centering
  \begin{tikzcd}[ampersand replacement=\&,column sep=large]
    \& \cf{X} \arrow[rd,"1"] \arrow[dd,rightharpoonup,xshift=1.5pt,"1"] \arrow[ld,rightharpoonup,"2"] \& \\
    \cf{Z} \arrow[ru,rightharpoonup,yshift=4pt,"3"] \arrow[rd,rightharpoonup,yshift=4pt,"1"] \& \& \cf{0} \arrow[ld,"1"] \\
    \& \cf{Y} \arrow[uu,rightharpoonup,xshift=-1.5pt,"2"] \arrow[lu,rightharpoonup,"3"] \&
  \end{tikzcd}
  \caption{A weakly reversible compartmental realization of Eq. \eqref{eq:opencc}}\label{fig:opencc1}
\end{figure}
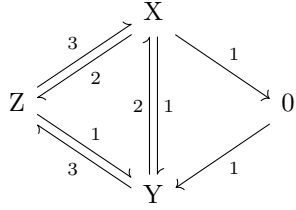
\end{example}

Let us see a nonlinear example.
\begin{example}\label{ex:opengcc}
The canonical realization of the \de\
\begin{equation}\label{eq:opengcc}
\dot{x}=-12x^3+6y^2,\quad\dot{y}=2x^3-4y^2+2
\end{equation}
is the irreversible, deficiency 3 reaction network consisting of two linkage classes
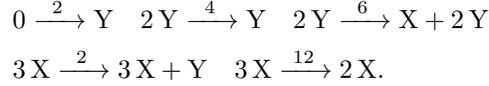
\begin{align*}
&\ce{0 ->[2] Y}\quad\ce{2Y ->[4] Y}\quad\ce{2Y ->[6] X + 2Y}\\
&\ce{3X ->[2] 3X + Y}\quad\ce{3X ->[12] 2X}.
\end{align*}
However, it is (much less) "obviously"
the \ikde\ of the weakly reversible, zero deficiency reaction network in Figure~\ref{fig:opengcc} consisting of a single linkage class.
\begin{figure}[!h]
  \centering
  \begin{tikzcd}[ampersand replacement=\&,column sep=small]
    \cf{3X} \arrow[rr,rightharpoonup,yshift=4pt,"1"] \arrow[rd,"3"'] \& \& \cf{2Y} \arrow[ll,rightharpoonup,yshift=1pt,"2"]\\
    \& \cf{0} \arrow[ur,"1"'] \& 
  \end{tikzcd}
  \caption{A weakly reversible, zero deficiency reaction network inducing Eq. \eqref{eq:opengcc}}\label{fig:opengcc}
\end{figure}
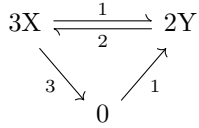
To show this in a formal way, we apply Theorem \ref{thm:general}.

With the notations of Theorem \ref{thm:general} we have
$$
\Zb=\left(\begin{array}{cc}
-{12} & 6 \\
{2} & -4
\end{array}\right),
\Yb=\left(\begin{array}{cc}
3 & 0 \\
0 & 2
\end{array}\right),
\bb=\left(\begin{array}{c}
0 \\
2
\end{array}\right).
$$
As $\Yb$ is invertible, $\vb=\Yb^{-1}\bb=\left(\begin{array}{c}
0 \\
1
\end{array}\right)$ should hold.
The unique solution of
\begin{align*}
\left(\begin{array}{cc}
3 & 0 \\
0 & 2
\end{array}\right)
\left(\begin{array}{cc}
-b_{21}-u_1 & b_{12} \\
b_{21} & -b_{12}-u_2
\end{array}\right)=
\left(\begin{array}{cc}
-{12} & 6 \\
{2} & -4
\end{array}\right)
\end{align*}
under the conditions
$b_{12},b_{21},u_1,u_2\ge0$
is  calculated with or without the direct use of $\Yb^{-1}$.
Thus, the unique reaction network to induce Eq. \eqref{eq:opengcc} (only having the complexes defined by the monomials on the right hand side) is seen in
Figure~\ref{fig:opengcc}.
\emph{How to write the linear equation for the coefficients and why does it tells us that the solution exists and is unique? Here we have four unknowns and four equations, in general we have $N^2=M^2$ equations and the same number of unknowns.}
\end{example}
Contrary to what the title suggests, here we show a network with thee complexes inducing a differential equation having two terms on the right hand side.
\begin{example}\label{ex:binomfour}
The \ikde\ of the network
\begin{equation*}
\ce{\alpha_1X <=>[k_1][k_{-1}] \beta_1X <=>[k_2][k_{-2}] \beta_2X}
\end{equation*}
with $\alpha_1<\beta_1<\beta_2$ and 
\begin{equation}
    k_2(\beta_2-\beta_1)=k_{-1}(\beta_1-\alpha_1)\label{eq:convexity}
\end{equation}
only has two terms on its right hand side:
\begin{equation*}
\dot{x}=k_1x^{\alpha_1}(\beta_1-\alpha_1)+k_{-2}x^{\beta_2}(\beta_1-\beta_2).
\end{equation*}
No wonder, as the condition \eqref{eq:convexity} implies that $\beta_1$ is the convex combination of $\alpha_1$ and $\beta_2.$

\end{example}

\begin{example}
The \ikde\ of the reaction network
\begin{equation}\label{eq:revreac}
\ce{2X + 5Y + Z <=>[$k_1$][$k_{-1}$] 3X + 4Y + Z}
\end{equation}
is \eqref{eq:poly} which we repeat here:
\begin{equation}
\dot{x}=k_1x^2y^5z-k_{-1}x^3y^4z\quad
\dot{y}=-k_1x^2y^5z+k_{-1}x^3y^4z\quad
\dot{z}=0
\end{equation}
This is of the form
\begin{equation}
\dot{\xb}=
k_1x^2y^5z\left(\begin{array}{c} 1 \\-1 \\0\end{array}\right)
+k_{-1}x^3y^4z\left(\begin{array}{c}-1 \\1 \\0\end{array}\right),
\end{equation}
therefore an inducing reaction network is just the one in  \eqref{eq:revreac}.
\end{example}

%%%%%%%%%%%%%%%
%%%%%%%%%%%%%%%%%%%%%%%%%%%%%%%%%%%%%%%%%%%%%%%%%%%%%%%%%%%%%%%%%%%%%%%
\subsection*{Mass conserving realizations}
%%%%%%%%%%%%%%%%%%%%%%%%%%%%%%%%%%%%%%%%%%%%%%%%%%%%%%%%%%%%%%%%%%%%%%%
The reaction \eqref{eq:reaction} is \df{mass conserving}, if there exists a vector $\varrhob\in \rr_{>0}^M$ such that for all $r=1,2,\dots,R$
\begin{equation}\label{eq:masscons}
\sum_{m=1}^M\alpha(m,r)\varrho^m=\sum_{m=1}^M\beta(m,r)\varrho^m
\end{equation}
holds.
Note that to ensure mass conservation the existence of a positive vector in the left kernel of the span of the range of the \rhs\ is necessary but not sufficient,~\cite[p. 89]{feinberghornkinetic}.

Consider again the differential equation
\begin{equation}\label{eq:de}
\dot{\xb}=\Zb\xb^\Yb.
\end{equation}
According to the definition \eqref{eq:masscons} the existence of a vector $\varrhob\in \rr_{> 0}^M$ such that
\begin{equation}\label{eq:massconsnecess}
\varrhob\T\Zb=\nulb\in\R^R
\end{equation}
is a necessary condition of mass conservation.
However, it is only sufficient with some restrictions.
\begin{theorem}\label{thm:masscons}
If \eqref{eq:de} has a realization such that the number of linkage classes is equal to the number of ergodic components (terminal strong linkage classes) and there exists a vector $\varrhob\in \rr_{> 0}^M $ such that \eqref{eq:massconsnecess} holds, then this realization is mass conserving.
\end{theorem}
It is shown in~\cite{Rudan2013b} that the same kinetic model may have both mass conserving and non-mass conserving realizations.
%CAN AN ODE HAVE MASS CONSERVING AND NON-MASSCONSERVING REALIZATION AT THE SAME TIME?

\end{document}